\documentclass[sn-mathphys-num]{sn-jnl}

\usepackage{graphicx}%
\usepackage{multirow}%
\usepackage{amsmath,amssymb,amsfonts}%
\usepackage{amsthm}%
\usepackage{mathrsfs}%
\usepackage[title]{appendix}%
\usepackage{xcolor}%
\usepackage{textcomp}%
\usepackage{manyfoot}%
\usepackage{booktabs}%
\usepackage{algorithm}%
\usepackage{algorithmicx}%
\usepackage{algpseudocode}%
\usepackage{listings}%

\newtheorem{theorem}{Theorem}[section]
\newtheorem{corollary}[theorem]{Corollary}

\newtheorem{proposition}[theorem]{Proposition}
\theoremstyle{definition}
\newtheorem{definition}[theorem]{Definition}
\newtheorem{example}[theorem]{Example}
\theoremstyle{remark}
\newtheorem{Remark}[theorem]{Remark}
\numberwithin{equation}{section}

\begin{document}

\title[A Survey on Constructing Parseval Fusion Frames via Scaling Weights]{A Survey on Constructing Parseval Fusion Frames via Scaling Weights}


\author[1]{\fnm{Ehsan} \sur{Ameli}}\email{eh.ameli@hsu.ac.ir}

\author*[1]{\fnm{Ali Akbar} \sur{Arefijamaal}}\email{arefijamaal@hsu.ac.ir}
\equalcont{These authors contributed equally to this work.}

\author[2,1]{\fnm{Fahimeh} \sur{Arabyani Neyshaburi}}\email{fahimeh.arabyani@gmail.com}
\equalcont{These authors contributed equally to this work.}

\affil[1]{\orgdiv{Department of Mathematics and Computer Sciences}, \orgname{Hakim Sabzevari University}, \state{Sabzevar}, \country{Iran}}

\affil[2]{\orgdiv{Department of Mathematical Sciences}, \orgname{Ferdowsi University of Mashhad}, \city{Mashhad}, \country{Iran}}


\abstract{The construction of Parseval fusion frames is highly desirable in a wide range of signal processing applications. In this paper, we study the problem of modifying the weights of a fusion frame in order to generate a Parseval fusion frame. To this end, we extend the notion of the scalability to the fusion frame setting. We then proceed to characterize scalable fusion Riesz bases and $1$-excess fusion frames. Furthermore, we provide the necessary and sufficient conditions for the scalability of certain $k$-excess fusion frames, $k\geq 2$. Finally, we present several pertinent examples to confirm the obtained results.}




\maketitle

\section{Introduction}\label{sec1}

Over the past few years, the construction of tight frames has emerged as a key concept in various areas of applied mathematics, computer science, and engineering. There have been many works to construct tight frames \cite{Cahill-Fickus,Casazza17,dual scalable frames,Gitta 13,convex geometry}. However, it is preferable to generate tight frames by simply scaling each frame vector, as this modification has no effect on the frame properties such as structure and erasure resilience. Thus, scaling appears to be a natural method for generating a Parseval frame. By the nature of scaling, we are looking for some scalars such that the scaled frame becomes a Parseval frame \cite{Gitta 13}. These types of frames possess a reconstruction formula analogous to orthonormal bases, and therefore they are easily used in a wide variety of applications such as sampling, signal processing, filtering, image processing, and in the other areas \cite{A.A.SH,Balazs,Casazza17,Gitta 13, convex geometry}.

Fusion frames are a generalization of ordinary frames in separable Hilbert spaces, first proposed by Casazza and Kutyniok in \cite{frame of subspace}. They are a novel development that provide a mathematical framework in many applications, including sensor networks, coding theory, filter bank theory, signal and image processing and wireless communications and many other fields, where cannot be modeled by ordinary frames \cite{excess of fusion,Weaving Hilbert space,Norm retrieval algorithms,Arabyani dual,A.A.SH,CasazzaFickus}.
One of the crucial properties that Parseval fusion frames have in common with fusion orthonormal bases is their useful role in signal processing. Indeed, they ensure a straightforward reconstruction of a signal and are resilient to noise \cite{CasazzaFickus}.
In this respect, the construction of Parseval fusion frames has considerable significance. Inspired by this perspective, our objective is to extend the concept of scalability to the fusion frame setting and to generate Parseval fusion frames.
A Parseval fusion frame $\mathcal{W}=\{(W_{i},\omega_{i})\}_{i \in I}$ for $\mathcal {H}$ has a property that every vector $f \in \mathcal {H}$ can be recovered via the reconstruction formula: $f=\sum_{i \in I}\omega_{i}^{2}\pi_{W_{i}}f.$ If a fusion frame is not Parseval, the reconstruction formula depends on the inverse of the fusion frame operator. This may be difficult or even impossible to compute.
Now, a central question emerges in the context of fusion frame theory is, how can a Parseval fusion frame be constructed from a given fusion frame without altering the structure of its subspaces?

In this survey, we propose an approach based on scaling weights to address this pivotal question. Moreover, if there exist weights that can turn a given fusion frame into a Parseval fusion frame, we refer to the fusion frame as weight-scalable. It is imperative to note that the weight-scalability process preserves the structure of subspaces, which aligns with the concept of scalability in frame theory. This provides further insight into the analysis and construction of Parseval fusion frames. Furthermore, we observe how the weights can be determined from those of the original fusion frame prior to construction. A natural question to ask is whether the weight-scalability in fusion frame theory enjoys similar properties as scalability in ordinary frames, and whether or not it provides a possibility for convenient reconstruction of signals. In addition, under what conditions is a fusion frame weight-scalable? This article responds to all of these questions. 

Our recent work \cite{excess of fusion} shows that in order to determine the excess of a fusion frame, it is sufficient to compute the excess of its local frame obtained from Riesz bases. Building upon this result, we are interested in investigating the weight-scalability of overcomplete fusion frames. Indeed, it is worthwhile to identify the conditions under which a $k$-excess fusion frame is weight-scalable for any $k \in \Bbb N.$

The remainder of the paper is organized as follows. After introducing fusion frames and reviewing known results on them in Section 2, our main result in Section 3 presents a way to construct Parseval fusion frames; so called the weight-scalability. We examine the weight-scalability of fusion frames and present results that are comparable to those valid for scalable frames. Furthermore, we completely characterize weight-scalable fusion Riesz bases and 1-excess fusion frames. In Section 4, we investigate the conditions under which some 2-excess fusion frames are not weight-scalable and give several examples in order to confirm the obtained results. Finally, in Section 5, we turn our attention to the weight-scalability of certain $k$-excess fusion frames for $k\geq2$ and derive an equivalent condition for the weight-scalability of 2-excess fusion frames by utilizing the acquired results.

\section{Preliminaries and Notations} 
We briefly review the concept of fusion frames and recall some of their main properties. We also present some results that will be used later in the paper. Throughout this paper, we suppose that $\mathcal {H}$ is a separable Hilbert space and $\mathcal {H}_n$ is an $n$-dimensional Hilbert space. Furthermore, $I$ and $J$ denote countable index sets and $I_{\mathcal {H}}$ denotes the identity operator on $\mathcal {H}$. We denote the set of all bounded operators on $\mathcal {H}$ by $B(\mathcal {H})$ and the orthogonal projection from Hilbert space $\mathcal {H}$ onto a closed subspace $ W \subseteq \mathcal {H}$ by $ \pi_{W}$. Moreover, we denote the null space of $T \in B(\mathcal {H})$ by $N(T)$. 
 \begin{definition}\cite{frame of subspace}
Let $\{W_{i}\}_{i \in I} $ be a family of closed subspaces of $\mathcal {H}$ and $ \{\omega_{i}\}_{i \in I} $ a family of weights, i.e.  $\omega_{i}>0 $, $ i \in I $. Then $ \mathcal{W}=\{(W_{i},\omega_{i})\}_{i \in I} $ is called a fusion frame for $\mathcal {H}$ if there exists the constants $0<A\leq B<\infty$ such that 
\begin{equation}\label{fusion frame def}
A \Vert f\Vert^{2}\leq \sum _{i \in I} \omega_{i}^2\Vert \pi_{W_{i}}f \Vert ^{2} \leq B\Vert f \Vert ^{2}, \quad(f \in \mathcal{H}).
\end{equation}
\end{definition}
The constants $A$ and $B$ are called the \textit{fusion frame bounds}. If we only have the upper bound in \eqref{fusion frame def}, then $ \mathcal{W} $ is said to be a \textit{fusion Bessel sequence}. A fusion frame is called \textit{A-tight} if $ A=B $, and \textit{Parseval} if $ A=B=1 $. If $ \omega_{i}=\omega $ for all $ i\in I $, then $ \mathcal{W} $ is called \textit{$ \omega $-uniform} and if $\textnormal{dim}W_{i}=n$ for all $ i\in I$, then $ \mathcal{W} $ is called \textit{n-equi-dimensional} fusion frame. We abbreviate 1-uniform fusion frames as $\{W_{i}\}_{i \in I}$. A family of closed subspaces $\{W_{i}\}_{i \in I} $ is said to be a \textit{fusion orthonormal basis} when $\mathcal {H}$ is the orthogonal sum of the subspaces $ W_{i}$ and it is a \textit{Riesz decomposition} of $\mathcal {H}$, if for every $ f \in \mathcal{H} $ there is a unique choice of $ f_{i}\in W_{i} $ such that $ f=\sum _{i \in I} f_{i} $. It is obvious that every fusion orthonormal basis is a Riesz decomposition for $\mathcal {H}$, and also every Riesz decomposition is a 1-uniform fusion frame for $\mathcal {H}$ \cite{frame of subspace}. Moreover, a family $\{W_{i}\}_{i \in I} $ of closed subspaces of $\mathcal {H}$ is a fusion orthonormal basis if and only if it is a 1-uniform Parseval fusion frame \cite{frame of subspace}. A fusion frame is said to be \textit{exact}, if it ceases to be a fusion frame whenever any of its elements is deleted. A family of closed subspaces $\{W_{i}\}_{i \in I} $ is called a \textit{fusion Riesz basis} whenever it is complete for $\mathcal {H}$ and there exist positive constants $C$ and $D$ such that for every finite subset $ J\subset I $ and arbitrary vector $ f_{j}\in W_{j}~ (j \in J)$, we have
\begin{equation*}\label{fusion Riesz basis}
C \sum _{j \in J}\Vert f_{j} \Vert ^{2} \leq \bigg\Vert \sum _{j \in J} \omega_{j} f_{j}\bigg\Vert ^2 \leq D\sum _{j \in J}\Vert f_{j} \Vert ^{2}.
\end{equation*}
Recall that for each sequence $\{W_{i}\}_{i \in I}$ of closed subspaces in $\mathcal {H}$, the space
\begin{equation*}
\sum_{i \in I}\bigoplus W_{i}=\left\lbrace \{f_{i}\}_{i \in I}: f_{i}\in W_{i}~,~ \sum_{i \in I}\Vert f_{i} \Vert ^{2}< \infty \right\rbrace , 
\end{equation*}
with the inner product
\begin{equation*}
\big\langle  \{f_{i}\}_{i \in I} , \{g_{i}\}_{i \in I} \big\rangle  = \sum_{i \in I}\langle f_{i} , g_{i}\rangle ,
\end{equation*}
constitutes a Hilbert space. For a Bessel sequence $\mathcal{W}=\{(W_{i},\omega_{i})\}_{i \in I}$, the \textit{synthesis operator} \hbox{$ T_{\mathcal{W}}: \sum_{i \in I}\bigoplus W_{i} \rightarrow \mathcal {H}$} is defined by 
\begin{equation*}
T_{\mathcal{W}}(\{f_{i}\}_{i \in I})=\sum_{i \in I}\omega_{i}f_{i}, \quad \{f_{i}\}_{i \in I} \in \sum_{i \in I}\bigoplus W_{i}.
\end{equation*}
Moreover, the \textit{fusion frame operator} $S_{\mathcal{W}}:\mathcal {H} \rightarrow \mathcal {H}$ defined by
\begin{equation*}
S_{\mathcal{W}}f=T_{\mathcal{W}}T^{*}_{\mathcal{W}}f=\sum_{i \in I}\omega_{i}^{2}\pi_{W_{i}}f,
\end{equation*}
is positive and self-adjoint. It is shown that for a fusion frame $\mathcal{W}$, the fusion frame operator $S_{\mathcal{W}}$ is invertible, and thus we have the following reconstruction formula \cite{frame of subspace}: 
\begin{equation*}
f=\sum_{i \in I}\omega_{i}^{2}S_{\mathcal{W}}^{-1}\pi_{W_{i}}f, \quad(f \in \mathcal {H}).
\end{equation*}
In \cite{frame of subspace}, it has been proved that $\mathcal{W}=\{(W_{i},\omega_{i})\}_{i \in I} $ is a Parseval fusion frame if and only if $S_{\mathcal{W}}=I_{\mathcal {H}}$. The family $\left\lbrace (S_{\mathcal{W}}^{-1}W_{i},\omega_{i})\right\rbrace _{i \in I}$, which is also a fusion frame, is called the \textit{canonical dual} of $\mathcal{W}$. Generally, a Bessel sequence  $\{(V_{i},\upsilon_{i})\}_{i \in I}$ is said to be an \textit{alternate (G\~{a}vru\c{t}a) dual} of $\mathcal{W}$, whenever
\begin{equation*}
 f=\sum_{i \in I}\omega_{i}\upsilon_{i}\pi_{V_{i}}S_{\mathcal{W}}^{-1}\pi_{W_{i}}f,\quad(f \in \mathcal {H}).
\end{equation*}
Let $\mathcal{W}=\{(W_{i},\omega_{i})\}_{i \in I} $ be a fusion frame for $\mathcal {H}$. A fusion Bessel sequence $\mathcal{V}=\{(V_{i},\upsilon_{i})\}_{i \in I}$ is a dual of $\mathcal{W}$ if and only if \cite{Osgooei}
\begin{equation*}
T_{\mathcal{V}}\varphi_{\mathcal{VW}} T_{\mathcal{W}}^{*}=I_{\mathcal {H}},
\end{equation*}
where the bounded operator $\varphi_{\mathcal{VW}}:\sum_{i \in I}\bigoplus W_{i} \rightarrow \sum_{i \in I}\bigoplus V_{i} $ is given by
\begin{equation*}
\varphi_{\mathcal{VW}}\big(\{f_{i}\}_{i \in I}\big)=\left\lbrace \pi_{V_{i}}S_{\mathcal{W}}^{-1}f_{i}\right\rbrace _{i \in I} .
\end{equation*}

For the basic facts about fusion frames we refer the reader to \cite{Arabyani dual,frame of subspace, 16, Excess 1,Nga,Rahimi}. The following result establishes the relationship between local and global properties.

\begin{theorem}\label{Local theorem}
\cite{frame of subspace}  Let $\{W_{i}\}_{i \in I}$ be a family of closed subspaces of $\mathcal {H}$, $ \omega _{i} >0 $ and $\{f_{i,j}\}_{j \in J_{i}}$ be a frame (Riesz basis) for $W_{i}$ with frame bounds $A_{i}$ and $B_{i}$ such that
\begin{equation*}
0 < A=\textnormal{inf}_{i \in I} A_{i}\leq \textnormal{sup}_{i \in I} B_{i}=B < \infty .
\end{equation*}
Then the following conditions are equivalent:
\begin{itemize}
\item[(i)] $\{(W_{i},\omega_{i})\}_{i \in I}$ is a fusion frame (fusion Riesz basis) for $\mathcal {H}$,
\item[(ii)] $\{\omega_{i}f_{i,j}\}_{i \in I, j \in J} $ is a frame (Riesz basis) for $\mathcal {H}$.
\end{itemize}
\end{theorem}

Finally, we list two known results that will be utilized in the subsequent sections.
\begin{proposition}\cite{16} \label{Gavruta1} 
Let $U \in B(\mathcal{H})$ and $ W\subset \mathcal{H} $ be a closed subspace. Then the following are equivalent:
\begin{itemize}
\item[(i)] $\pi_{UW}U=U\pi_{W}.$
\item[(ii)]$U^{*}UW\subseteq W .$
\end{itemize}
\end{proposition}

\begin{proposition}\cite{frame of subspace, mitra sh.} \label{mitra sh.} 
Let $\mathcal{W}=\{(W_{i},\omega_{i})\}_{i \in I} $ be a fusion frame for $\mathcal {H}$. Then the following are equivalent:
\begin{itemize}
\item[(i)] $\mathcal{W}$ is a fusion Riesz basis.
\item[(ii)] $S_{\mathcal{W}}^{-1}W_{i}\perp W_{j}$ for all $i , j \in I,~ i \neq j $.
\item[(iii)] $ \omega_{i}^{2}\pi_{W_{i}}S_{\mathcal{W}}^{-1}\pi_{W_{j}}=\delta_{i,j}\pi_{W_{j}}$ for all $i , j \in I $. 
\end{itemize}
\end{proposition}

\section{Weight-Scalable Fusion Riesz Bases and 1-excess Fusion Frames}
In this section, we present a method for generating Parseval fusion frames that relies on preserving subspaces and merely scaling weights. First, we examine the results that are comparable to those observed for scalable frames. Then we provide a characterization for the scalability of fusion Riesz bases and discuss on the scalability of $1$-excess fusion frames. In the following, we state the general definition of the weight-scalability for fusion frames.

\begin{definition}\label{gama scalability fusion}
A fusion frame $\mathcal{W}=\{(W_{i},\omega_{i})\}_{i \in I} $ of $\mathcal {H}$ is called \textit{weight-scalable} if there exists a family of weights $\gamma:=\{\gamma_{i}\}_{i \in I}$ such that $\mathcal{W}_{\gamma}:=\{(W_{i},\omega_{i}\gamma_{i})\}_{i \in I}$ is a Parseval fusion frame for $\mathcal {H}$.
\end{definition}

We use $\gamma$-scalable whenever we intend to refer to the family of weights $\gamma=\{\gamma_{i}\}_{i \in I}$ that make a fusion frame weight-scalable. Assume that $\{W_{i}\}_{i \in I}$ is a family of closed subspaces of $\mathcal {H}$. Obviously, any finite fusion frame $\{W_{i}\}_{i=1}^{n} $ together with any sequence of weights $\{\omega_{i}\}_{i=1}^{n} $ is again a fusion frame. This may not be the case in the infinite-dimensional situation. However, if $\{\omega_{i}\}_{i \in I}$ is a semi-normalized sequence \cite{Balazs}, then $\mathcal{W}=\{(W_{i},\omega_{i})\}_{i \in I}$ is a fusion frame. Corresponding to a sequence of weights $\gamma:=\{\gamma_{i}\}_{i \in I},$ we define the diagonal operator $D_{\gamma}$ in $\sum_{i \in I}\bigoplus W_{i}$ as follows.
\begin{equation*}
D_{\gamma}\{f_{i}\}_{i \in I}=\{\gamma_{i}f_{i}\}_{i \in I},~\{f_{i}\}_{i \in I} \in \sum_{i \in I}\bigoplus W_{i},
\end{equation*}
where
\begin{equation*}
\textnormal{dom}~D_{\gamma}:=\left\lbrace \{f_i\}_{i\in I}\in \sum_{i \in I}\bigoplus W_{i}~:~ \{\gamma_i f_i\}_{i\in I}\in \sum_{i \in I}\bigoplus W_{i} \right\rbrace .
\end{equation*}
In \cite{Rahimi scalable}, it was shown that $D_{\gamma}$ is a self-adjoint operator in $\sum_{i \in I}\bigoplus W_{i}$. Moreover, an equivalent condition for the scalability of fusion frames was provided, see \cite[Proposition 2.9]{Rahimi scalable} for more details.

\begin{proposition}\label{main scalability}\cite{Rahimi scalable}
Let $\mathcal{W}=\{(W_{i},\omega_{i})\}_{i \in I} $ be a fusion frame for $\mathcal {H}$. Then the following are equivalent: 
\begin{itemize}
\item[(i)] $\mathcal{W}$ is scalable,
\item[(ii)] There exists a positive diagonal operator $D_{\gamma}$ in $\sum_{i \in I}\bigoplus W_{i}$ such that
\begin{equation}\label{yyy}
\overline{T_{\mathcal{W}}D_{\gamma}}D_{\gamma}T^{*}_{\mathcal{W}}=I_{\mathcal {H}}.
\end{equation}
\end{itemize}
\end{proposition}

It is evident that 1-equi-dimensional fusion frames correspond to ordinary frames, and thereby Proposition \ref{main scalability} will align with Proposition 2.4 of \cite{Gitta 13}. Furthermore, if $\liminf_{i \in I}\omega_{i}>0,$ then the operator $D_{\gamma}$ as in Proposition \ref{main scalability}, is bounded and \eqref{yyy} can be simplified as  $T_{\mathcal{W}}D^{2}_{\gamma}T^{*}_{\mathcal{W}}=I_{\mathcal {H}}.$ A straightforward consequence of this characterization is that the weight-scalability, like ordinary frames, is invariant under unitary transformations. What is more, if $U \in B(\mathcal {H})$ satisfies $U^*UW_{i}\subseteq W_{i}~(i \in I)$, then it follows from Proposition \ref{Gavruta1} that
\begin{equation*}
S_{U\mathcal{W}_{\gamma}}=\sum _{i \in I} (\omega_{i}\gamma_{i})^2 \pi_{UW_{i}}=US_{\mathcal{W}_{\gamma}}U^{-1}.
\end{equation*}
Therefore, $\mathcal{W}$ is weight-scalable if and only if $U\mathcal{W}$ is weight-scalable. 

\subsection{Fusion Riesz bases}

In the following, we provide a characterization for weight-scalable fusion Riesz bases.

\begin{proposition}\label{new def}
Let $\mathcal{W}=\{(W_{i},\omega_{i})\}_{i \in I} $ be a fusion Riesz basis for $\mathcal {H}$ and $\omega=\{\omega_{i}\}_{i \in I}$. Then the following are equivalent:
\begin{itemize}
\item[(i)] $\mathcal{W}$ is $\omega^{-1}$-scalable,
\item[(ii)] $S_{\mathcal{W}}^{-1}\pi_{W_{i}}=\omega_{i}^{-2}\pi_{W_{i}},$ for all $i \in I$,
\item[(iii)] $\mathcal{W}$ is orthogonal.
\end{itemize}
\end{proposition}

\begin{proof}
$(i)\Rightarrow (ii)$ If $\mathcal{W}$ is $\omega^{-1}$-scalable, then $\mathcal{W}_{\omega^{-1}}$ is a Parseval fusion frame. Hence,
\begin{align*}
S_{\mathcal{W}}^{-1}\pi_{W_{i}}&=S_{\mathcal{W}_{\omega^{-1}}}S_{\mathcal{W}}^{-1}\pi_{W_{i}}
\\&=\sum_{j \in I}\pi_{W_{j}}S_{\mathcal{W}}^{-1}\pi_{W_{i}}
\\&=\pi_{W_{i}}S_{\mathcal{W}}^{-1}\pi_{W_{i}}=\omega_{i}^{-2}\pi_{W_{i}},
\end{align*}
for all $i \in I$, where the last line is obtained by Proposition \ref{mitra sh.}. The implications $(ii)\Rightarrow (iii)$ and $(iii)\Rightarrow (i)$ are straightforward.
\end{proof}

The following example demonstrates a weight-scalable fusion Riesz basis.

\begin{example}
Consider $W_{1}=\textnormal{span} \{u\}$ and $W_{2}=\{0\}\times \mathcal {H}_{2},$ where $u=(u_{1},u_{2},u_{3}) \in \mathcal {H}_{3}$ is a unit vector such that $u_{1}\neq 0.$ Then $\mathcal{W}=\{(W_{i},\omega_{i})\}_{i=1}^2$ is a fusion Riesz basis for $\mathcal {H}_{3}$ with the fusion frame operator
\begin{equation*}
S_{\mathcal{W}}=
\begin{pmatrix}
\omega_{1}^2u^2_{1}   &    \omega_{1}^2 u_{1}u_{2}     &    \omega_{1}^2 u_{1}u_{3} \\
\omega_{1}^2u_{1}u_{2}    &   \omega_{1}^2u_{2}^2+\omega_{2}^2  &    \omega_{1}^2 u_{2}u_{3} \\
\omega_{1}^2u_{1}u_{3}     &      \omega_{1}^2 u_{2}u_{3}      &     \omega_{1}^2u_{3}^2+\omega_{2}^2 \\
\end{pmatrix}.
\end{equation*}
Obviously, $\mathcal{W}$ is $\omega^{-1}$-scalable if and only if $u_{2}=u_{3}=0,$ which confirms Proposition \ref{new def}.
\end{example}

A direct computation analogous to the above example confirms that, without losing the generality, we may consider 1-uniform fusion frames. However, it should be noted that our results are valid for general fusion frames.

\subsection{1-excess fusion frames}

In what follows, we study the weight-scalability of 1-excess fusion frames. First, we state the concept of excess for fusion frames \cite{Excess 1}. Let $\mathcal{W}=\{(W_{i},\omega_{i})\}_{i \in I} $ be a fusion frame for $\mathcal {H}$ with the synthesis operator $T_{\mathcal{W}}$. The \textit{excess} of $\mathcal{W}$ is defined as 
\begin{equation*}
e(\mathcal{W})=\textnormal{dim}N(T_{\mathcal{W}}).
\end{equation*}

In \cite{excess of fusion}, the authors present a novel approach for the excess of fusion frames, which is related to the excess of their local frames.
\begin{proposition}\cite{excess of fusion}
Let $\mathcal{W}=\{(W_{i},\omega_{i})\}_{i \in I} $ be a fusion frame for $\mathcal {H}$ with a local frame $\mathcal{F}=\{\omega_{i}f_{i,j}\}_{i \in I, j \in J_{i}},$ where $\{f_{i,j}\}_{j \in J_{i}}$ is a Riesz basis for $W_{i},$ for all $i \in I$. Then 
\begin{equation*}
e(\mathcal{W})=e(\mathcal{F}).
\end{equation*}
\end{proposition}

Suppose that $\mathcal{W}=\{W_{i}\}_{i \in I} $ is a fusion frame for $\mathcal {H}$. As stated before, if all $W_{i}$'s are one dimensional, then $\mathcal {W}$ corresponds to an ordinary frame. Accordingly, the notion of scalability for such frames and fusion frames becomes equivalent. For instance, 1-equi-dimensional fusion frames associated with 1-excess tight frames, such as the Mercedes-Benz frame, can easily be made weight-scalable. In this respect, it is crucial to investigate the weight-scalability of fusion frames in which the subspaces are not necessarily one dimensional. Taking this into account, we discuss the weight-scalability of 1-excess fusion frames and provide some equivalent conditions for the weight-scalability of such fusion frames.

\begin{theorem}\label{main theorem gamma scalable}
Let $\mathcal{V}$ be a 1-excess fusion frame for $\mathcal {H}$ with the Riesz part $\mathcal{W}.$ Then the following conditions are equivalent:
\begin{itemize}
\item[(i)] $\mathcal{V}$ is $\gamma$-scalable, $\gamma=\{\gamma_{i}\}_{i=0}^{\infty}$.
\item[(ii)] Every excess element $x$ belongs to a one dimensional subspace and 
\begin{equation}\label{www}
\gamma_{0}^{2}\pi_{W_{i}}S_{\mathcal{W}_{\gamma}}^{-1}\pi_{V_{0}}\pi_{W_{j}} =\left( \gamma_{i}^{-2}\delta_{i,j}-\pi_{W_{i}}\right)\pi_{W_{j}}, \quad (i , j \geq 1),
\end{equation}
where $V_{0}$ is the subspace generated by $x.$
\item[(iii)] Every excess element $x$ belongs to a one dimensional subspace and 
\begin{equation}\label{simplify}
\gamma_{0}^{2}\left\langle \pi_{W_{j}}f,x \right\rangle x_{i}=\left( \delta_{i,j}-\gamma_{i}^{2}\pi_{W_{i}}\right)\pi_{W_{j}}f,\quad (i , j \geq 1,~f\in \mathcal {H}).
\end{equation}
\end{itemize}
\end{theorem}
 
\begin{proof}
$(i)\Rightarrow(ii)$ Towards a contradiction, assume that the excess element $x$ does not belong to a one dimensional subspace. Thus, without losing the generality, $\mathcal{V}$ can be expressed as 
\begin{equation*}
\mathcal{V}= V_{1} \cup \{W_{i}\}_{i =2}^{\infty},
\end{equation*}
where $\mathcal{W}=\{W_{i}\}_{i =1}^{\infty}$ is a fusion Riesz basis, $V_{1}=W_{1}\oplus\textnormal{span}\{x\}$ and the redundant element $x =\sum_{i=2}^{\infty} x_{i}$ is a unit vector of $\mathcal {H}$ such that $x\perp W_{1}$ and $x_{i} \in W_{i}$ for each $i \geq 2$. Since $\mathcal{V}$ is $\gamma$-scalable, then there exists a sequence of weights $\gamma=\{\gamma_{i}\}_{i=1}^\infty$ such that  
\begin{equation}\label{e3e}
f=S_{\mathcal{V_{\gamma}}}f=\gamma_{1}^{2}\pi_{W_{1}}f+\gamma_{1}^{2}\left\langle f,x \right\rangle x +\sum_{i =2}^{\infty}\gamma_{i}^{2}\pi_{W_{i}}f, \quad ( f\in \mathcal {H}).
\end{equation}
Put a non zero vector $ g\in W_{1}$ in \eqref{e3e}. It gives $g=\gamma_{1}^{2}g+\sum_{i =2}^{\infty}\gamma_{i}^{2}\pi_{W_{i}}g.$
It follows from the Riesz decomposition property of $\mathcal{W}$ that $\gamma_{1}=1$ and $W_{1}\perp W_{i}$ for all $i \geq 2.$ 
Moreover, by reformulating \eqref{e3e}, we get
\begin{equation*}
S_{\mathcal{W}_{\gamma}}f=f-\gamma_{1}^{2}\left\langle f,x \right\rangle x=f-\left\langle f,x \right\rangle x,
\end{equation*}
for each $f \in \mathcal {H}$. Substituting $x$ for $f$ gives $S_{\mathcal{W}_{\gamma}}x=0.$ Thus, $x=0.$ That is a contradiction and so $x$ must belong to a one dimensional subspace. As such, $\mathcal{V}$ must be rewritten as 
\begin{equation*}
\mathcal{V}= V_{0} \cup \mathcal{W},
\end{equation*}
where $V_{0}=\textnormal{span}\{x\}$ and $x =\sum_{i=1}^{\infty} x_{i}$ is a unit vector of $\mathcal {H}$. In light of $\mathcal{V}$ is $\gamma$-scalable, we get $\gamma_{0}^{2}\pi_{V_{0}}+S_{\mathcal{W}_{\gamma}}= I_{\mathcal {H}},$ and hence
\begin{equation*}\label{e4e}
S_{\mathcal{W}_{\gamma}}^{-1}=\gamma_{0}^{2}S_{\mathcal{W}_{\gamma}}^{-1}\pi_{V_{0}}+ I_{\mathcal {H}}.
\end{equation*}
Thus, for each $i , j \geq 1$, we obtain
\begin{align*}
\gamma_{0}^{2}\pi_{W_{i}}S_{\mathcal{W}_{\gamma}}^{-1}\pi_{V_{0}}\pi_{W_{j}}&=\gamma_{0}^{2}\pi_{W_{i}}S_{\mathcal{W}_{\gamma}}^{-1}\pi_{V_{0}}\pi_{W_{j}} +\pi_{W_{i}}\pi_{W_{j}}-\pi_{W_{i}}\pi_{W_{j}}
\\&=\pi_{W_{i}}\left( \gamma_{0}^{2}S_{\mathcal{W}_{\gamma}}^{-1}\pi_{V_{0}}+ I_{\mathcal {H}} \right)  \pi_{W_{j}}-\pi_{W_{i}}\pi_{W_{j}}
\\&=\pi_{W_{i}}S_{\mathcal{W}_{\gamma}}^{-1}\pi_{W_{j}}-\pi_{W_{i}}\pi_{W_{j}}
\\&=\left( \gamma_{i}^{-2}\delta_{i,j}-\pi_{W_{i}}\right) \pi_{W_{j}},
\end{align*}
where the last equality follows from Proposition \ref{mitra sh.}. So, (ii) holds.

$(ii)\Rightarrow(i)$ Assume that $\mathcal{V}= V_{0} \cup \mathcal{W},$ where $V_{0}=\textnormal{span}\{x\}.$ In addition, suppose that there exists a sequence of weights $\gamma=\{\gamma_{i}\}_{i=0}^{\infty}$ such that \eqref{www} is satisfied. Then, we induce
\begin{align*}
S_{\mathcal{V}_{\gamma}}S_{\mathcal{W}}&=\left( \gamma_{0}^{2}\pi_{V_{0}}+S_{\mathcal{W}_{\gamma}}\right)  S_{\mathcal{W}}
\\&=\sum_{j=1}^{\infty}\left(\gamma_{0}^{2}S_{\mathcal{W}_{\gamma}}S_{\mathcal{W}_{\gamma}}^{-1}\pi_{V_{0}}\pi_{W_{j}}+S_{\mathcal{W}_{\gamma}}\pi_{W_{j}} \right) 
\\&=\sum_{i,j=1}^{\infty}\left((\gamma_{0}\gamma_{i})^{2}\pi_{W_{i}}S_{\mathcal{W}_{\gamma}}^{-1}\pi_{V_{0}}\pi_{W_{j}}+\gamma_{i}^{2}\pi_{W_{i}}\pi_{W_{j}} \right) 
\\&=\sum_{i,j=1}^{\infty}\delta_{i,j}\pi_{W_{j}}=\sum_{i=1}^{\infty}\pi_{W_{i}}=S_{\mathcal{W}},
\end{align*}
where the last line is derived by applying \eqref{www}. Therefore, $S_{\mathcal{V}_{\gamma}}=I_{\mathcal {H}},$ i.e. $\mathcal{V}$ is $\gamma$-scalable.

$(ii)\Leftrightarrow(iii)$ The equation \eqref{www} may be simplified as follows:
\begin{align*}
\left( \delta_{i,j}-\gamma_{i}^{2}\pi_{W_{i}}\right)\pi_{W_{j}}f&=(\gamma_{i}\gamma_{0})^{2}\pi_{W_{i}}S_{\mathcal{W}_{\gamma}}^{-1}\pi_{V_{0}}\pi_{W_{j}} 
\\&=(\gamma_{i}\gamma_{0})^{2}\left\langle \pi_{W_{j}}f,x \right\rangle \pi_{W_{i}}S_{\mathcal{W}_{\gamma}}^{-1}x 
\\&=\gamma_{0}^{2}\left\langle \pi_{W_{j}}f,x \right\rangle x_{i},
\end{align*}
where the last equality follows from Proposition \ref{mitra sh.}. This completes the proof.
\end{proof}

This result can be generalized to a certain type of $k$-excess fusion frames by a similar approach for any $k\geq 2,$ see Proposition \ref{1-excess general}. In the following, we examine the validity of Theorem \ref{main theorem gamma scalable} through an example.

\begin{example}\label{1-excess example}
Let $\{e_{i}\}_{i \in \Bbb Z}$ be an orthonormal basis for $\mathcal {H}$, $W_{1}=\overline{\textnormal{span}}_{i\geq 0}\{e_{i}\}$ and $W_{2}=\overline{\textnormal{span}}_{i\leq 0}\{e_{i}\}$. It is easily observed that $\mathcal{V}=\{W_{i}\}_{i=1}^2$ is an exact fusion frame which is not Riesz basis, see \cite{frame of subspace}. In addition, it is proved in \cite{excess of fusion} that $\mathcal{V}$ is a 1-excess fusion frame for $\mathcal {H}.$ According to Theorem \ref{main theorem gamma scalable}, $\mathcal{V}$ is not weight-scalable, as the excess element does not belong to a one dimensional subspace. However, we investigate this matter directly. Suppose $\mathcal{V}$ is $\gamma$-scalable, then there exists a sequence of weights $\gamma=\{\gamma_{i}\}_{i \in \Bbb Z}$ such that   
\begin{align*}
f=S_{\mathcal{V}_{\gamma}}f&=\sum_{i\geq 0}\gamma_{1}^{2}\left\langle f,e_{i} \right\rangle e_{i}+ \sum_{i\leq 0}\gamma_{2}^{2}\left\langle f,e_{i} \right\rangle e_{i}
\\&=\sum_{i>0}\gamma_{1}^{2}\left\langle f,e_{i} \right\rangle e_{i}+ \sum_{i<0}\gamma_{2}^{2}\left\langle f,e_{i} \right\rangle e_{i}+(\gamma_{1}^{2}+\gamma_{2}^{2})\left\langle f,e_{0} \right\rangle e_{0},
\end{align*}
for all $f \in \mathcal {H}$. Substituting $e_{-1}, e_{0}, e_{1}$ for $f$ results in $\gamma_{i}=1~(i=1,2)$ and $\gamma_{1}^{2}+\gamma_{2}^{2}=1,$ a contradiction. 
\end{example}

In the sequel, in view of Theorem \ref{main theorem gamma scalable}, we proceed to investigate the weight-scalability of 1-excess fusion frames of the form
\begin{equation}\label{every 1 excess}
\mathcal{V}= V_{0} \cup \mathcal{W},
\end{equation}
where $\mathcal{W}=\{W_{i}\}_{i=1}^{\infty}$ is a fusion Riesz basis, $V_{0}=\textnormal{span}\{x\}$ and $x =\sum_{i \in \sigma} x_{i}$ is a unit vector of $\mathcal {H}$ such that $x_{i} \in W_{i}$ and $\sigma=\{i \geq 1 \mid x_{i}\neq0 \}$. The next result provides the essential criteria for the weight-scalability of $\mathcal{V}$. 

\begin{corollary}\label{erer}
Let $\mathcal{V}$ be the 1-excess fusion frame given by \eqref{every 1 excess}. If $\mathcal{V}$ is weight-scalable, then the following statements hold: 
\begin{itemize}
\item[(i)] $\gamma_{0}< 1$ and $x$ is an eigenvector of $S_{\mathcal{W}_\gamma}$ associated with eigenvalue $1-\gamma_{0}^2$.
\item[(ii)] $\left\langle x_{j},x \right\rangle =\dfrac{1-\gamma_{j}^{2}}{\gamma_{0}^{2}},~\pi_{W_{i}}x_{j}=\dfrac{\gamma_{j}^{2}-1}{\gamma_{i}^{2}}x_{i}$ for $ j \in \sigma,~i\neq j.$ 
\item[(iii)] If $ j \in \sigma^{c},$ then $\gamma_{j}=1,~x \perp W_{j}$ and $W_{j}\perp W_{i}$ for $i\neq j.$ 
\item[(iv)] If $ j\in \sigma,$ then $\gamma_{j}\neq 1,~x \not\perp W_{j}$ and $\textnormal{dim}W_{j}=1$.
\end{itemize}
\end{corollary}

\begin{proof}
Due to $\mathcal{V}$ is weight-scalable, there exists a sequence of weights $\gamma=\{\gamma_{i}\}_{i=0}^\infty$ such that 
\begin{equation}\label{yty}
f=\gamma_{0}^{2}\left\langle f,x \right\rangle x +S_{\mathcal{W}_{\gamma}}f, \quad  (f\in \mathcal {H}).
\end{equation}

(i) Putting $f=x$  in \eqref{yty}, we get $S_{\mathcal{W}_{\gamma}}x=(1-\gamma_{0}^{2})x.$ If $\gamma_{0}=1,$ it yields $x=0,$ a contradiction. Moreover, the positivity of $S_{\mathcal{W}_{\gamma}}$ assures that  $\gamma_{0}<1.$ Thus (i) holds.

(ii) Replacing $f=x_{j}~(j \in \sigma)$ in \eqref{simplify} of Theorem \ref{main theorem gamma scalable}, we get $\left( \delta_{i,j}-\gamma_{i}^{2}\pi_{W_{i}}\right)x_{j}=\gamma_{0}^{2}\left\langle x_{j},x \right\rangle x_{i},$ which leads to
\begin{align*}
\begin{cases}
\left( 1-\gamma_{0}^{2}\left\langle x_{j},x \right\rangle -\gamma_{j}^{2}\right) x_{j}=0, \quad (j=i),
\\
\gamma_{0}^{2}\left\langle x_{j},x \right\rangle x_{i}+\gamma_{i}^{2}\pi_{W_{i}}x_{j}=0, \quad (j\neq i).
\end{cases}
\end{align*}
The first equation yields $\left\langle x_{j},x \right\rangle =\dfrac{1-\gamma_{j}^{2}}{\gamma_{0}^{2}}.$ Moreover, the second equation implies that $\pi_{W_{i}}x_{j}=\dfrac{\gamma_{j}^{2}-1}{\gamma_{i}^{2}}x_{i}$ for all $i\neq j.$

(iii) Setting $0 \neq g \in W_{j}$ in \eqref{simplify}, we have
\begin{equation}\label{3.333}
\left( \delta_{i,j}-\gamma_{i}^{2}\pi_{W_{i}}\right)g=\gamma_{0}^{2}\left\langle g,x\right\rangle x_{i},\quad (i\geq 1).
\end{equation} 
Since $x_{j}=0,$ then by assigning $i=j$ in \eqref{3.333} yields $\gamma_{j}=1.$ Moreover, after renaming the indices $i \leftrightarrow j$ and taking $0 \neq f \in W_{i}$ in \eqref{simplify}, we obtain
\begin{equation*}
\left( \delta_{j,i}-\pi_{W_{j}}\right)f=\gamma_{0}^{2}\left\langle f,x \right\rangle x_{j}=0, \quad (i\neq j).
\end{equation*}
It derives that $\pi_{W_{j}}f=0,$ i.e., $W_{j}\perp W_{i}$ for $i\neq j.$ Hence, $x \perp W_{j}$ because
$\left\langle x,g \right\rangle =\sum_{i \in \sigma}\left\langle x_{i},g \right\rangle =0$ for each $0 \neq g \in W_{j}.$

(iv) Let $x_{j}\neq0.$ By contrary, assume that $x\perp W_{j},$ then 
\begin{equation}\label{lll}
\Vert x_{j} \Vert ^2+\sum_{i\in \sigma , i\neq j}\left\langle x_{i},x_{j} \right\rangle =\left\langle x,x_{j} \right\rangle =0. 
\end{equation}
If $i\neq j$ in \eqref{3.333}, we get $W_{j}\perp W_{i}.$ Thus, it yields from \eqref{lll} that $\Vert x_{j} \Vert =0$ and thereby $x_{j}=0,$ a contradiction. So, $x \not\perp W_{j}.$ That means that there exists a non zero vector $g \in W_{j}$ such that $\left\langle g,x \right\rangle \neq 0.$ Putting $f=g$ in \eqref{3.333} and considering $i=j,$ we get 
\begin{equation*}
(1-\gamma_{j}^{2})g=\gamma_{0}^{2}\left\langle g,x\right\rangle x_{j}.
\end{equation*}
Hence, it follows that $\gamma_{j}\neq 1$ and $g \in\textnormal{span}\{x_{j}\}.$ We claim that $W_{j}=\textnormal{span}\{x_{j}\}.$ In fact, if there exists $0 \neq  h \in W_{j}$ such that $\left\langle h,x \right\rangle=0,$ then from (iii) we obtain $\gamma_{j}=1$, that is a contradiction. Therefore, $\textnormal{dim}W_{j}=1,$ the desired result.
\end{proof}

This result prompts us to characterize all weight-scalable 1-excess fusion frames.

\begin{theorem} 
Every weight-scalable 1-excess fusion frame $\mathcal{V}$ for $\mathcal{H}$ is precisely of the form
\begin{equation*}
\mathcal{V}=\left\lbrace \textnormal{span}\left\lbrace f_{i} \right\rbrace \right\rbrace _{i \in I_{1}}\cup \left\lbrace W_{i} \right\rbrace _{i \in I_{2}},
\end{equation*}
where $\mathcal{F}:={\left\lbrace f_{i} \right\rbrace _{i \in I_{1}}}$ is a strictly scalable 1-excess frame sequence of unit vectors and $\left\lbrace W_{i} \right\rbrace _{i \in I_{2}}$ is an orthogonal fusion Riesz basis for $\left( \textnormal{span}\left\lbrace f_{i} \right\rbrace _{i \in I_{1}}\right) ^{\perp}$.
\end{theorem}

\begin{proof}
Assume that $\mathcal{F}$ is a strictly scalable frame for some $\mathcal{H}_{1}\subset \mathcal{H}$, then there exists a sequence of weights $\{\gamma_{i}\} _{i \in I_{1}}$ such that 
\begin{equation}\label{yes}
\sum_{i \in I_{1}}\gamma_{i}^{2}\left\langle f,f_{i} \right\rangle f_{i}=f, \quad (f \in \mathcal{H}_{1}).
\end{equation}
Moreover, $\left\lbrace W_{i} \right\rbrace _{i \in I_{2}}$ is a uniform scalable fusion Riesz basis for $\mathcal{H}_{1}^{\perp},$ i.e., $\sum_{i \in I_{2}}\pi_{W_{i}}=I_{\mathcal{H}_{1}^{\perp}}.$ We put $\gamma_{i}=1$ for all $i \in I_{2}.$ Thus, by considering $\gamma=\{\gamma_{i}\}_{i \in I_{1}\cup I_{2}}$ and in view of the fact that for every $g \in \mathcal{H},$ there exist unique vectors $g_{1} \in \mathcal{H}_{1}$ and $g_{2} \in \mathcal{H}_{1}^{\perp}$ such that $g=g_{1}+g_{2}$ we obtain 
\begin{align*}
S_{\mathcal{V}_{\gamma}}g&=\sum_{i \in I_{1}}\gamma_{i}^{2}\pi_{\textnormal{span}\{f_{i}\}}(g_{1}+g_{2})+\sum_{i \in I_{2}}\pi_{W_{i}}(g_{1}+g_{2})
\\&=\sum_{i \in I_{1}}\gamma_{i}^{2}\left\langle g_{1},f_{i} \right\rangle f_{i}+\sum_{i \in I_{2}}\pi_{W_{i}}g_{2}
\\&=g_{1}+g_{2}=g,
\end{align*}
where the last line is obtained by using \eqref{yes}.
Conversely, if the 1-excess fusion frame $\mathcal{V}$ given by \eqref{every 1 excess}  is weight-scalable, then there exists a sequence of weights $\gamma=\{\gamma_{i}\}_{i \in I}$ such that $S_{\mathcal{V}_{\gamma}}=I_{\mathcal{H}}.$ The conditions $(iii)$ and $(iv)$ of Corollary \ref{erer} ensure that 
$\mathcal{F}:=\left\lbrace x,x_{i} \right\rbrace _{i \in \sigma}$ constitutes a 1-excess ordinary frame for $\mathcal{H}_{1}:=\overline{\textnormal{span}}\left\lbrace x,x_{i} \right\rbrace _{i \in \sigma}$ and $\left\lbrace W_{i} \right\rbrace _{i \in \sigma^{c}}$ is an orthogonal fusion Riesz basis for $\mathcal{H}_{1}^{\perp}=\oplus_{i \in \sigma^{c}}W_{i}$. Hence, it follows from the weight-scalability of $\mathcal{V}$ and Corollary \ref{erer} that $\gamma_{i}=1$ for $i \in \sigma^{c}.$ So, for every $f \in \mathcal{H}_{1}$ we get
\begin{align*}
f&=S_{\mathcal{V}_{\gamma}}f
\\&=\gamma_{0}^{2}\pi_{\textnormal{span} \{x\}}f+\sum_{i \in \sigma}\gamma_{i}^{2}\pi_{\textnormal{span}\{x_{i}\}}f+\sum_{i \in \sigma^{c}}\pi_{W_{i}}f
\\&=\gamma_{0}^{2}\left\langle f,x \right\rangle x+\sum_{i \in \sigma}\gamma_{i}^{2}\frac{\left\langle f,x_{i} \right\rangle}{\Vert x_{i} \Vert ^{2}} x_{i}=S_{\gamma \mathcal{F}}f.
\end{align*}
Therefore, $\mathcal{F}$ is strictly scalable.
\end{proof}

Now, we present an equivalent condition for the weight-scalability of 1-excess fusion frames containing orthogonal fusion Riesz bases.

\begin{corollary}\label{OfR}
Let $\mathcal{V}$ be the 1-excess fusion frame as defined in \eqref{every 1 excess} and $\mathcal{W}$ an orthogonal fusion Riesz basis for $\mathcal {H}$. Then the following are equivalent:
\begin{itemize}
\item[(i)] $\mathcal{V}$ is weight-scalable.
\item[(ii)] There exists a subspace of $\mathcal{W}$ generated by the excess element.
\end{itemize}
\end{corollary}

\begin{proof}
$(i)\Rightarrow(ii)$ One can easily see that there exists a $j\geq 1$ such that the excess element $x \not\perp W_{j}.$ Since $\mathcal{V}$ is weight-scalable, it gives $W_{j}=\textnormal{span}\{x_{j}\}$ by Corollary \ref{erer}(iv). Putting $f=x_{j}$ in \eqref{simplify} yields
\begin{equation*}
\gamma_{0}^{2} \left\langle x_{j},x \right\rangle x_{i}=\left( \delta_{i,j}-\gamma_{i}^{2}\pi_{W_{i}}\right)x_{j}=0, \quad (i \neq j).
\end{equation*}
It can be seen that $ \left\langle x_{j},x \right\rangle$ is a non zero scalar because of the orthogonality of $\mathcal{W}.$ Therefore, $x_{i}=0$ for all $i \neq j,$ which implies that $W_{j}=V_{0}.$

$(ii)\Rightarrow(i)$ Assume that there exists a $j\geq 1$ such that $W_{j}=V_{0}.$ Taking $\gamma_{0},\gamma_{j}$ satisfying $ \gamma_{0}^{2}+\gamma_{j}^{2}=1$ and $\gamma_{i}=1$ for $i\neq j$, we conclude that
\begin{align*}
S_{\mathcal{V}_{\gamma}}&=\left( \gamma_{0}^{2}+\gamma_{j}^{2}\right) \pi_{W_{j}}+\sum_{i =1,i\neq j}^{\infty}\gamma_{i}^{2}\pi_{W_{i}}
\\&=\pi_{W_{j}}+\sum_{i =1,i\neq j}^{\infty}\pi_{W_{i}}=\pi_{\oplus_{i =1}^{\infty}W_{i}}= I_{\mathcal{H}}.
\end{align*} 
\end{proof}

Motivated by the obtained results, we seek a more precise understanding of weight-scalable 1-excess fusion frames. To this end, consider the 1-excess fusion frame $\mathcal{V}$ introduced in \eqref{every 1 excess}. 
It is apparent that if $\textnormal{dim}W_{i}>1$ for all $i \geq1,$ then $\mathcal{V}$ is not weight-scalable. Indeed, the redundant element belongs to one of $W_{i}$'s, which is in conflict with the assertion made in Theorem \ref{main theorem gamma scalable}. Therefore, there is at least a one dimensional subspace in weight-scalable 1-excess fusion frames.

\begin{Remark}
(i) The orthogonality condition of $\mathcal{W}$ can be removed from the assumptions of Corollary \ref{OfR}, provided that $\textnormal{dim}\mathcal {H}=3$ and $\mathcal{W}$ is not a 1-equi-dimensional fusion Riesz basis. In this case, the weight-scalability of $\mathcal{V}$ implies the orthogonality of $\mathcal{W}$. Indeed, without losing the generality, such 1-excess fusion frames for $\mathcal {H}_{3}$ are of the form $\mathcal{V}= V_{0} \cup \mathcal{W},$
where
\begin{equation*}
V_{0}=\textnormal{span}\{v\},~W_{1}=\textnormal{span}\{u\},~W_{2}=\{0\}\times \mathcal {H}_{2},
\end{equation*}
such that $v=(v_{1},v_{2},v_{3}),~u=(u_{1},u_{2},u_{3})$ and $u_{1}\neq 0.$ Let $\mathcal{V}$ be weight-scalable and by contrary assume that $W_{1} \not\perp W_{2}.$ If $V_{0}=W_{1},$ then $V_{0} \not\perp W_{2}.$ Hence, $\textnormal{dim}W_{2}=1,$ by Corollary \ref{erer}(iv), a contradiction. Thus, suppose that $v$ and $u$ are linearly independent. Now, if $\left\lbrace  v, u, e_{2} \right\rbrace $ or $\left\lbrace  v, u, e_{3} \right\rbrace $ constitutes a basis for $\mathcal {H}_3,$ then $e_{3}$ or $e_{2}$ may be regarded as the excess element. Thus, Theorem \ref{main theorem gamma scalable} guarantees that it must belong to a one dimensional subspace, which is a contradiction. Otherwise, if $\left\lbrace  v, u, e_{2} \right\rbrace $ and $\left\lbrace  v, u, e_{3} \right\rbrace $ are not bases, it easily follows that $e_{2},e_{3} \in \textnormal{span}\{v,u\},$ as $v$ and $u$ are linearly independent. Hence, $\mathcal{W}$ is not a Riesz basis. That is again a contradiction. Therefore, $\mathcal{W}$ is orthogonal and $V_{0}=W_{1},$ by Corollary \ref{OfR}.

(ii) The assertion presented in (i) is not universally applicable to the cases where $\textnormal{dim}\mathcal {H}\neq 3$ or $\mathcal {W}$ is 1-equi-dimensional. To illustrate, one can refer to 1-equi-dimensional fusion frames associated with the Mercedes-Benz frames. As another example, let $\{e_{i}\}_{i=1}^{n}$ be the canonical orthonormal basis for $\mathcal {H}_{n}~(n\geq 4).$ Consider
\begin{align*}
V_{0}&=\textnormal{span}\left\lbrace \alpha e_{1}-\sqrt{1-\alpha ^{2}}e_{2}\right\rbrace ,
\\
W_{1}&=\textnormal{span}\{e_{1}\},~W_{2}= \textnormal{span}\{e_{1}+e_{2}\},~W_{3}= \textnormal{span}\{e_{i}\}_{i=3}^{n},
\end{align*}
where $ 0<\alpha <\frac{\sqrt{2}}{2}$ and $\{W_{i}\}_{i=1}^{3}$ is a non-orthogonal fusion Riesz basis. Then $\mathcal{V}= V_{0} \cup \{W_{i}\}_{i =1}^{3}$ is a 1-excess fusion frame for $\mathcal {H}_{n}.$ A straightforward computation shows that $\mathcal{V}$ is $\gamma$-scalable by
\begin{align*}
\begin{cases}
\gamma_{0}=\sqrt{\dfrac{1}{1-\alpha ^{2}+\alpha \sqrt{1-\alpha ^{2}}}},
~\gamma_{1}=\sqrt{\dfrac{1-2\alpha ^{2}}{1-\alpha ^{2}+\alpha \sqrt{1-\alpha ^{2}}}},
\\
\gamma_{2}=\sqrt{\dfrac{2\alpha \sqrt{1-\alpha ^{2}}}{1-\alpha ^{2}+\alpha \sqrt{1-\alpha ^{2}}}},~\gamma_{3}=1.
\\
\end{cases}
\end{align*}
\end{Remark}

\begin{example} 
Let $\{e_{i}\}_{i=1}^4$ be the canonical orthonormal basis for $\mathcal {H}_4.$ Take
\begin{align*}
V_{0}&=\textnormal{span}\{\alpha_{1}e_{1}+\alpha_{2}e_{2}+\alpha_{3}e_{3}+\alpha_{4}e_{4}\},
\\
W_{1}&=\textnormal{span}\{e_{1},e_{2}\},~W_{2}= \textnormal{span}\{e_{3}+\beta e_{1}\},~W_{3}= \textnormal{span}\{e_{4}\},
\end{align*}
where $\alpha_{i},\beta \in \Bbb C$ for $1\leq i \leq 4$ and $\mathcal{W}=\{W_{i}\}_{i=1}^{3}$ is a fusion Riesz basis. Then
$\mathcal{V}= V_{0} \cup \{W_{i}\}_{i =1}^{3}$ is a 1-excess fusion frame for $\mathcal {H}_4.$ If $\beta=0,$ then $\mathcal{W}$ is orthogonal, and it can be easily seen that $\mathcal{V}$ is weight-scalable if and only if $V_{0}=W_{2}$ or  $V_{0}=W_{3},$ according to Corollary \ref{OfR}. Otherwise, let $\beta\neq 0.$ If either $\alpha_{1}$ or $\alpha_{2}$ is non zero, then $\mathcal{V}$ is not weight-scalable, by Corollary \ref{erer}(iv). Also, in the case $\alpha_{1}=\alpha_{2}=0,$ Corollary \ref{erer}(iii) confirms that $\mathcal{V}$ is not weight-scalable.
\end{example}


\section{Weight-Scalable 2-excess Fusion Frames}
 
This section is devoted to studying the weight-scalability of 2-excess fusion frames. In general, four distinct situations there exist for the excess elements of a 2-excess fusion frame. However, this investigation focuses on the situation where the excess elements are located in more than one subspace of the Riesz part, as this case allows for the derivation of outcomes that are comparable to those of the other three cases. It should be noted that a 2-excess fusion frame is not weight-scalable if the redundant elements lie in a subspace of its Riesz part, see Corollary \ref{k-excess in one subspace} for the general case. Hence, without losing the generality, every 2-excess fusion frame $\mathcal{V}$ in which the redundant vectors are located in different subspaces of its Riesz part can be expressed as
\begin{equation}\label{every 2 excess}
\mathcal{V}=\left\lbrace V_{1},V_{2},W_{i}\right\rbrace _{i =3}^{\infty},
\end{equation}
where $V_{\ell}= W_{\ell}\oplus\textnormal{span}\{x^{(\ell)}\}$ for $\ell=1,2,$ and $x^{(1)} =\sum_{i=2}^{\infty} x^{(1)}_{i}$ and $x^{(2)} =\sum_{i=1,i\neq2}^{\infty} x^{(2)}_{i}$ are unit vectors of $\mathcal{H}$ such that $x^{(\ell)}_{i}$'s are vectors of $W_{i}$ and $\mathcal{W}=\{W_{i}\}_{i=1}^\infty$ is a fusion Riesz basis for $\mathcal{H}.$ At first, we present some equivalent conditions for the weight-scalability of $\mathcal{V}.$

\begin{proposition}\label{main 2-excess f.f.}
Let $\mathcal{V}$ be the $2$-excess fusion frame defined by \eqref{every 2 excess} and $\mathcal{W}$ a fusion Riesz basis for $\mathcal {H}.$ Then the following conditions are equivalent:
\begin{itemize}
\item[(i)] $\mathcal{V}$ is $\gamma$-scalable, $\gamma=\{\gamma_{i}\}_{i=1}^{\infty}$.
\item[(ii)] $\sum_{\ell =1}^{2}\gamma_{\ell}^{2}\pi_{W_{i}}S_{\mathcal{W}_{\gamma}}^{-1}\pi_{\textnormal{span}\{x^{(\ell)}\}}\pi_{W_{j}} =\left( \gamma_{i}^{-2}\delta_{i,j}-\pi_{W_{i}}\right)\pi_{W_{j}},$ for all $i , j.$
\item[(iii)] $\sum_{\ell =1}^{2}\gamma_{\ell}^{2}\left\langle \pi_{W_{j}}f,x^{(\ell)} \right\rangle x^{(\ell)}_{i} =\left( \delta_{i,j}-\gamma_{i}^{2}\pi_{W_{i}}\right)\pi_{W_{j}}f,$
for all $f\in \mathcal {H}.$
\end{itemize}
\end{proposition}

\begin{proof}
The proof is derived by using an analogous approach to Theorem \ref{main theorem gamma scalable}.
\end{proof}

In order to simplify the notation, the symbols $x$ and $y$ are used instead of $x^{(1)}$ and $x^{(2)},$ as the excess elements introduced in \eqref{every 2 excess}. One of our aims is to characterize weight-scalable 2-excess fusion frames. In this regard, we obtain some necessary conditions for the weight-scalability of such fusion frames. \color{black}

\begin{theorem}\label{one dim}
Let the 2-excess fusion frame $\mathcal{V}$ given by \eqref{every 2 excess} be weight-scalable. Then the following statements hold:
\begin{itemize}
\item[(1)] $\left\langle x,y \right\rangle =0,~\pi_{W_{i}}x=\dfrac{1-\gamma_{1}^{2}}{\gamma_{i}^{2}}x_{i}$ for all $i\geq 2.$ 

\item[(2)] $\pi_{W_{i}}y=\dfrac{1-\gamma_{2}^{2}}{\gamma_{i}^{2}}y_{i}$ for all $i\geq 2.$

\item[(3)] $ \gamma_{1},\gamma_{2} \neq 1,~W_{1}\perp W_{2}$ and $\textnormal{dim}W_{i}=1$ for $i=1,2.$ 

\item[(4)] $\left\langle x_{2},x \right\rangle =\dfrac{1-\gamma_{2}^{2}}{\gamma_{1}^{2}},~\pi_{W_{i}}x_{2}=\dfrac{\gamma_{2}^{2}-1}{\gamma_{i}^{2}}x_{i}$ for all $i\geq 3.$

\item[(5)] $\left\langle y_{1},y \right\rangle =\dfrac{1-\gamma_{1}^{2}}{\gamma_{2}^{2}},~\pi_{W_{i}}y_{1}=\dfrac{\gamma_{1}^{2}-1}{\gamma_{i}^{2}}y_{i}$ for all $i\geq 3.$

\item[(6)] $\Vert x_{2} \Vert ^2=\Vert y_{1} \Vert ^{-2}=\dfrac{\gamma_{2}^{2}(1-\gamma_{2}^{2})}{\gamma_{1}^{2}(1-\gamma_{1}^{2})}.$

\item[(7)] $1\leq \gamma_{1}^{2}+\gamma_{2}^{2} <2.$
\end{itemize}
\end{theorem}

\begin{proof}
Since $\mathcal{V}$ is $\gamma$-scalable, then there exists a sequence of weights $\gamma=\{\gamma_{i}\}_{i=1}^\infty$ such that   
\begin{align}\label{00100}
f=\gamma_{1}^{2}\pi_{V_{1}}f+\gamma_{2}^{2}\pi_{V_{2}}f+\sum_{i =3}^{\infty}\gamma_{i}^{2}\pi_{W_{i}}f, \quad  (f\in \mathcal {H}).
\end{align}
First of all, it turns out that $y_{1},x_{2}\neq 0.$ To see this, we only prove that $x_{2}\neq 0$ and the proof of $y_{1}\neq 0$ is analogous. To the contrary, assume that $x_{2}=0.$ Putting a non zero vector $ f \in W_{2}$ in \eqref{00100} and using the Riesz decomposition property of $\mathcal{W},$ we get $(1-\gamma_{2}^{2})f=0.$ So, $\gamma_{2}=1.$ We now investigate two cases: If $y_{1}=0,$ it similarly follows $\gamma_{1}=1.$ Moreover, by taking $f=x$ in \eqref{00100}, we get
\begin{equation*}
\pi_{V_{2}}x+\sum_{i=3}^{\infty}\gamma_{i}^{2}\pi_{W_{i}}x =x-\pi_{V_{1}}x=0.
\end{equation*}
Hence, $\Vert \pi_{V_{2}}x \Vert ^2+\sum_{i=3}^{\infty}\gamma_{i}^{2}\Vert \pi_{W_{i}}x\Vert ^2 =\left\langle \pi_{V_{2}}x+\sum_{i=3}^{\infty}\gamma_{i}^{2}\pi_{W_{i}}x , x \right\rangle =0,$ which assures that $\pi_{V_{2}}x=0,$ and $\pi_{W_{i}}x=0$ for all $i\geq 3.$ Therefore, 
\begin{equation*}
\Vert x \Vert ^2=\left\langle x,\sum_{i=3}^{\infty} x_{i} \right\rangle =\sum_{i=3}^{\infty} \left\langle x,x_{i} \right\rangle =0. 
\end{equation*}
This means that $x=0,$ which is a contradiction. Otherwise, suppose that $y_{1}\neq0.$ Substituting $f=x$ in \eqref{00100} gives $x=\gamma_{1}^{2}x+\pi_{W_{2}}x+\left\langle x,y\right\rangle y+\sum_{i =3}^{\infty}\gamma_{i}^{2}\pi_{W_{i}}x.$ By employing the Riesz decomposition property of $\mathcal{W},$ we get $\left\langle x,y\right\rangle y_{1}=0.$ Since $y_{1}\neq 0,$ then $\left\langle x,y\right\rangle=0.$ Thus, by setting $f=y$ in \eqref{00100}, it yields
\begin{align*}
y&=y+\gamma_{1}^{2}\pi_{W_{1}}y+\gamma_{1}^{2}\left\langle y,x\right\rangle x+\sum_{i =3}^{\infty}\gamma_{i}^{2}\pi_{W_{i}}y
\\&=y+\gamma_{1}^{2}\pi_{W_{1}}y+\sum_{i =3}^{\infty}\gamma_{i}^{2}\pi_{W_{i}}y.
\end{align*}
So, we obtain $\pi_{W_{i}}y=0$ for all $i \neq 2.$ Since $y \perp W_{2},$ we induce that $y=0,$ which is again a contradiction. That is why $x_{2}\neq 0.$ Now, we proceed to prove the theorem.

(1) Taking $f=x$ in \eqref{00100} leads to
\begin{align*}
x=\gamma_{1}^{2}x+\gamma_{2}^{2}\pi_{W_{2}}x+\gamma_{2}^{2}\left\langle x,y\right\rangle y+\sum_{i =3}^{\infty}\gamma_{i}^{2}\pi_{W_{i}}x.
\end{align*}
By an analogous method to the first part of the proof, we get $\left\langle x,y \right\rangle=0.$ Furthermore, $x_{i}-\gamma_{1}^{2}x_{i}-\gamma_{i}^{2}\pi_{W_{i}}x=0$ for $i\geq 2,$ which ensures that $\pi_{W_{i}}x=\dfrac{1-\gamma_{1}^{2}}{\gamma_{i}^{2}}x_{i}$ for all $i\geq 2.$ Condition (2) is proved similar to (1) by setting $f=y$ in \eqref{00100}.

(3) We first note that, $\gamma_{i}\neq 1$ for $i=1,2.$ Indeed, if $\gamma_{1}=1,$ then it follows from (1) that $x \perp W_{i}$ for all $i \geq 2.$ Since $x \perp W_{1},$ we get $x=0,$ which is a contradiction, similarly $\gamma_{2}\neq 1.$ Putting $j=1,~i=1,2$ and $ g \in W_{1}$ in Proposition \ref{main 2-excess f.f.}(iii), we obtain
$g-\gamma_{1}^{2}g-\gamma_{2}^{2}\left\langle g,y\right\rangle y_{1}=0$
and $\gamma_{2}^{2}\pi_{W_{2}}g=0.$ Hence, $g=\dfrac{\gamma_{2}^{2}\left\langle g,y\right\rangle}{1-\gamma_{1}^{2}}y_{1}$ and $W_{1}\perp W_{2}.$ In particular, $W_{1}=\textnormal{span}\{y_{1}\}$. Similarly, taking $j=i=2$ and $ f\in W_{2}$ in Proposition \ref{main 2-excess f.f.}(iii), gives $f=\dfrac{\gamma_{1}^{2}\left\langle f,x\right\rangle}{1-\gamma_{2}^{2}}x_{2},$ which yields that $W_{2}=\textnormal{span}\{x_{2}\}.$

(4) Substituting $f=x_{2}$ in \eqref{00100} and applying (3), we get
\begin{align*}
x_{2}&=\gamma_{1}^{2}\left\langle x_{2},y_{1}\right\rangle \frac{y_{1}}{\Vert y_{1}\Vert ^2}+\gamma_{1}^{2}\left\langle x_{2},x\right\rangle x+\gamma_{2}^{2}x_{2}+\sum_{i =3}^{\infty}\gamma_{i}^{2}\pi_{W_{i}}x_{2}
\\&=\gamma_{1}^{2}\left\langle x_{2},x\right\rangle x+\gamma_{2}^{2}x_{2}+\sum_{i =3}^{\infty}\gamma_{i}^{2}\pi_{W_{i}}x_{2}.
\end{align*}
It follows from the Riesz decomposition property of $\mathcal{W}$ that
\begin{align*}
\begin{cases}
\left( 1-\gamma_{1}^{2}\left\langle x_{2},x \right\rangle -\gamma_{2}^{2}\right) x_{2}=0,
\\
\gamma_{1}^{2}\left\langle x_{2},x \right\rangle x_{i}+\gamma_{i}^{2}\pi_{W_{i}}x_{2}=0, \quad (i\geq 3).
\end{cases}
\end{align*}
The first equation gives $\left\langle x_{2},x \right\rangle =\dfrac{1-\gamma_{2}^{2}}{\gamma_{1}^{2}}.$ In addition, the second equation yields that $\pi_{W_{i}}x_{2}=\dfrac{\gamma_{2}^{2}-1}{\gamma_{i}^{2}}x_{i}$ for all $i\geq 3.$

(5) Putting $f=y_{1}$ in \eqref{00100}, the result is obtained by an analogous approach to (4). Condition (6) is also easily proved by using (1), (2), (4) and (5).

(7) Equation \eqref{00100} can be reformulated as follows:
\begin{equation*}
S_{\mathcal{W}_{\gamma}}f=f-\gamma_{1}^{2}\left\langle f,x \right\rangle x-\gamma_{2}^{2}\left\langle f,y \right\rangle y, \quad  (f\in \mathcal {H}).
\end{equation*}
In (3), we observed that $\gamma_{i}\neq 1$ for $i=1,2.$ Now, according to the positivity of  $S_{\mathcal{W}_{\gamma}},$ we get 
\begin{equation*}
\gamma_{1}^{2}\vert \left\langle f,x \right\rangle\vert ^2 +\gamma_{2}^{2} \vert \left\langle f,y \right\rangle \vert ^2 \leq \Vert f \Vert ^2, \quad  (f\in \mathcal {H}).
\end{equation*}
Substituting $x$ and $y$ for $f$, it derives that $\gamma_{i}< 1$ for $i=1,2.$ Applying (2) and (5) we obtain
\begin{align*}
\sum_{i=3}^\infty \gamma_{i}^{-2} \Vert y_{i} \Vert ^2 &=\sum_{i=3}^\infty \gamma_{i}^{-2} \left\langle y_{i},y_{i} \right\rangle 
\\&=\sum_{i=3}^\infty \gamma_{i}^{-2} \left\langle \frac{\gamma_{i}^2}{1-\gamma_{2}^{2}}\pi_{W_{i}}y,\frac{\gamma_{i}^2}{\gamma_{1}^{2}-1}\pi_{W_{i}}y_{1} \right\rangle 
\\&=\frac{1}{\left( 1-\gamma_{2}^{2}\right) \left( \gamma_{1}^{2}-1\right) }\sum_{i=3}^\infty \gamma_{i}^{2} \left\langle \pi_{W_{i}}y,\pi_{W_{i}}y_{1} \right\rangle 
\\&=\frac{1}{\left( 1-\gamma_{2}^{2}\right) \left( \gamma_{1}^{2}-1\right) } \left\langle \sum_{i=3}^\infty \gamma_{i}^{2}\pi_{W_{i}}y,y_{1} \right\rangle .
\end{align*}
Replacing $f$ by $y$ in \eqref{00100} yields
\begin{equation}\label{1}
\begin{aligned}
\sum_{i=3}^\infty \gamma_{i}^{-2} \Vert y_{i} \Vert ^2 &=\frac{1}{\left( 1-\gamma_{2}^{2}\right) \left( \gamma_{1}^{2}-1\right) } \left\langle \sum_{i=3}^\infty \gamma_{i}^{2}\pi_{W_{i}}y,y_{1} \right\rangle
\\&=\frac{1-\left( \gamma_{1}^{2}+\gamma_{2}^{2}\right) }{(1-\gamma_{2}^{2})(\gamma_{1}^{2}-1)}\left\langle y,y_{1} \right\rangle 
\\&=\frac{\gamma_{1}^{2}+\gamma_{2}^{2}-1}{(1-\gamma_{2}^{2})(1-\gamma_{1}^{2})}\times \frac{1-\gamma_{1}^{2}}{\gamma_{2}^{2}}=\frac{\gamma_{1}^{2}+\gamma_{2}^{2}-1}{\gamma_{2}^{2}(1-\gamma_{2}^{2})}\geq 0.
\end{aligned}
\end{equation}
Since $\gamma_{i}< 1~(i=1,2),$ it easily follows from the last inequality that $\gamma_{1}^{2}+\gamma_{2}^{2} \geq 1.$ Therefore, we induce that $1\leq \gamma_{1}^{2}+\gamma_{2}^{2} <2,$ which completes the proof.
\end{proof}
\color{black}

In light of the proof of Theorem \ref{one dim}(7) and analogous to \eqref{1}, it can be concluded that 
\begin{equation}\label{2}
\sum_{i=3}^\infty \gamma_{i}^{-2} \Vert x_{i} \Vert ^2=\frac{\gamma_{1}^{2}+\gamma_{2}^{2}-1}{\gamma_{1}^{2}(1-\gamma_{1}^{2})}\geq 0.
\end{equation}
Thus, if the 2-excess fusion frame $\mathcal{V}$ is $\gamma$-scalable for some $\gamma$ such that $\gamma_{1}^{2}+\gamma_{2}^{2}=1,$ then it follows from \eqref{1} and \eqref{2} that $y_{i}=x_{i}=0$ for all $i \geq 3$, and thereby $ x\in W_{2}$ and $y\in W_{1}.$
In this case, we will obtain an equivalent condition for the weight-scalability of $\mathcal{V},$ see Corollary \ref{2-excess scalability}. Now, we provide two examples of 2-excess fusion frames that are weight-scalable.

\begin{example}\label{Counterexample}
(1) Let $\{e_{i}\}_{i=1}^3$ be the canonical orthonormal basis for $\mathcal {H}_3.$ Consider $\mathcal{V}=\left\lbrace W_{1} \oplus W_{2},W_{2} \oplus W_{1},W_{3}\right\rbrace ,$ where $W_{i}=\textnormal{span}\{e_{i}\}$ for $1\leq i \leq 3.$ Then $\mathcal{V}$ is a 2-excess fusion frame containing the fusion orthonormal basis $\{W_{i}\}_{i=1}^3.$ It is easy to see that $\mathcal{V}$ is $\gamma$-scalable by $\gamma_{i}=\sqrt{1/2}$ for $i=1,2$ and $\gamma_{3}=1.$

(2) Let $\{e_{i}\}_{i=1}^4$ be the canonical orthonormal basis for $\mathcal {H}_4.$ Take 
\begin{align*}
V_{1}&= W_{1} \oplus \textnormal{span}\left\lbrace \frac{1}{2}e_{2}+\frac{\sqrt{3}}{2}e_{4} \right\rbrace , 
\\
V_{2}&= W_{2} \oplus \textnormal{span}\left\lbrace \frac{1}{2}e_{1}+\frac{\sqrt{3}}{2}e_{3} \right\rbrace,
\\
W_{3}&= \textnormal{span}\left\lbrace \frac{1}{2}e_{2}-\frac{\sqrt{3}}{2}e_{4} \right\rbrace, ~ W_{4}= \textnormal{span}\left\lbrace \frac{1}{2}e_{1}-\frac{\sqrt{3}}{2}e_{3} \right\rbrace ,
\end{align*}
where $W_{i}=\textnormal{span}\{e_{i}\}$ for $i=1,2$ and $\mathcal{W}=\{W_{i}\}_{i=1}^4 $ is a fusion Riesz basis for $\mathcal {H}_4.$ Then, the 2-excess fusion frame $\mathcal{V}=\{V_{1},V_{2},W_{3},W_{4}\}$ is $\gamma$-scalable by $\gamma_{i}=\sqrt{2/3}$ for $1\leq i \leq 4.$ More precisely, 
\begin{align*}
S_{\mathcal{V}_{\gamma}}&=\sum_{i=1}^2\gamma_{i}^2\pi_{W_{i}}+\gamma_{1}^2\pi_{\textnormal{span}\left\lbrace \frac{1}{2}e_{2}+\frac{\sqrt{3}}{2}e_{4} \right\rbrace} +\gamma_{2}^2\pi_{\textnormal{span}\left\lbrace \frac{1}{2}e_{1}+\frac{\sqrt{3}}{2}e_{3} \right\rbrace}
\\&+\gamma_{3}^2\pi_{\textnormal{span}\left\lbrace \frac{1}{2}e_{2}-\frac{\sqrt{3}}{2}e_{4} \right\rbrace}+\gamma_{4}^2\pi_{\textnormal{span}\left\lbrace \frac{1}{2}e_{1}-\frac{\sqrt{3}}{2}e_{3} \right\rbrace}.
\end{align*}
Thus, for every $(a,b,c,d) \in \mathcal {H}_4,$ we get
\begin{align*}
S_{\mathcal{V}_{\gamma}}(a,b,c,d)&=2/3\bigg[ \left( a,b,0,0\right) +\left( 0,(b+\sqrt{3}d)/4,0,(\sqrt{3}b+3d)/4\right) 
\\&+\left((a+\sqrt{3}c)/4,0,(\sqrt{3}a+3c)/4,0\right) 
\\&+\left( 0,(b-\sqrt{3}d)/4,0,(3d-\sqrt{3}b)/4\right)
\\&+\left( (a-\sqrt{3}c)/4,0,(3c-\sqrt{3}a)/4,0\right) \bigg] =(a,b,c,d),
\end{align*}
as required.
\end{example}

In the preceding theorem, it was shown that the 2-excess fusion frame $\mathcal{V}$ is not weight-scalable whenever $x_{2}=0$ or $y_{1}=0.$ The following theorem examines the weight-scalability of $\mathcal{V}$ in other cases for $x$ and $y.$

\begin{theorem}\label{30}
Let $\mathcal{W}=\{W_{i}\}_{i=1}^\infty$ be a fusion Riesz basis for $\mathcal{H}.$ If $x_{i}=0~(y_{i}=0)$ for any $i \geq 3$ and $y_{j} \neq 0~(x_{j} \neq 0)$ for some $j \geq 3,$ then the $2$-excess fusion frame $\mathcal{V}$ introduced in \eqref{every 2 excess} is not weight-scalable.
\end{theorem}

\begin{proof}
For the sake of contradiction, assume that $\mathcal{V}$ is $\gamma$-scalable. Then there exists a sequence of weights $\gamma=\{\gamma_{i}\}_{i=1}^\infty$ such that \eqref{00100} holds. Suppose that $x_{i}=0~(i \geq 3),$ so $x_{2}\neq0.$ If $y_{1}= 0,$ then $\gamma_{1}=1,$ by the proof of Theorem \ref{one dim}. Putting $f=x$ in \eqref{00100} leads to $x=x+\gamma_{2}^{2} x+\sum_{i =3}^{\infty}\gamma_{i}^{2}\pi_{W_{i}}x.$ It follows from the Riesz decomposition property of $\mathcal{W}$ that $\gamma_{2}=0,$ a contradiction. Otherwise, assume that $y_{1}\neq 0.$ Since $x$ is a unit vector of $W_{2},$ Theorem \ref{one dim}(4) assures that $\gamma_{1}^2 +\gamma_{2}^2=1.$ In view of \eqref{1}, it implies that $y_{i}=0$ for all $i \geq 3$, which is in conflict with the assumption. Therefore, $\mathcal{V}$ is not weight-scalable. A similar result is obtained whenever $y_{i}=0~(i \geq 3),$ as desired.
\end{proof}

\begin{example}
Let $\{e_{i}\}_{i=1}^{\infty}$ be an orthonormal basis for $\mathcal {H}.$ Consider 
\begin{align*}
V_{1}&= \textnormal{span}\{e_{1},e_{2}\} \oplus \textnormal{span}\{e_{3}\},
\\
V_{2}&= \textnormal{span}\{e_{3}\} \oplus \textnormal{span}\{ \alpha_{1}e_{1}+\alpha_{2}e_{2}+\beta(\alpha_{3}e_{1}+\alpha_{4}e_{4})+\alpha_{7}e_{6}\},
\\
V_{3}&= \textnormal{span}\{\alpha_{3}e_{1}+\alpha_{4}e_{4},\alpha_{5}e_{2}+\alpha_{6}e_{5},e_{6}\},
\\
V_{4}&= \textnormal{span}\{e_{\ell}\}_{\ell \geq 7},
\end{align*}
where $ \alpha_{i}, \beta \in \Bbb C$ for all $1 \leq i \leq 7$ and $\alpha_{i}\neq 0$ for $i=4,6$. Then $\mathcal{V}=\{V_{i}\}_{i=1}^{4}$ is a 2-excess fusion frame for $\mathcal{H}$. Assume that $\mathcal{V}$ is $\gamma$-scalable, then there exists a sequence of weights $\gamma=\{\gamma_{i}\}_{i=1}^{4}$ such that 
\begin{equation}\label{last}
f=\sum_{i =1}^{4}\gamma_{i}^{2}\pi_{V_{i}}f, \quad  (f\in \mathcal {H}).
\end{equation}
Putting $f=e_{3}$ in \eqref{last} gives $\gamma_{1}^2 +\gamma_{2}^2=1$. Moreover, replacing $f$ by $e_{5}$ in \eqref{last} yields
\begin{align*}
e_{5}&=\gamma_{3}^{2}\pi_{V_{3}}e_{5}=\gamma_{3}^{2} \pi_{\textnormal{span}\{\alpha_{5}e_{2}+\alpha_{6}e_{5}\}}e_{5} 
\\&=\frac{\gamma_{3}^2}{\alpha_{5}^2+\alpha_{6}^2}\left\langle e_{5},\alpha_{5}e_{2}+\alpha_{6}e_{5} \right\rangle (\alpha_{5}e_{2}+\alpha_{6}e_{5})
=\frac{\gamma_{3}^2}{\alpha_{5}^2+\alpha_{6}^2}\left( \alpha_{5}\alpha_{6}e_{2}+\alpha_{6}^2e_{5}\right).
\end{align*}
Since $\alpha_{6}\neq 0,$ then we get $\alpha_{5}=0$. Now, by setting $f=e_{2}$ in \eqref{last}, we have
\begin{equation*}
\begin{aligned}
e_{2}&=\gamma_{1}^{2}\pi_{V_{1}}e_{2}+\gamma_{2}^{2}\pi_{V_{2}}e_{2}
\\&=\gamma_{1}^{2}e_{2}+\gamma_{2}^{2} \pi_{\textnormal{span}\{  (\alpha_{1}+\beta\alpha_{3}) e_{1}+\alpha_{2}e_{2}+\beta \alpha_{4}e_{4}+\alpha_{7}e_{6}\}}e_{2} 
\\&=\left( \gamma_{1}^{2}+\frac{\gamma_{2}^2 \alpha_{2}^2}{c}\right) e_{2}+\frac{\gamma_{2}^2\alpha_{2}}{c}\left(( \alpha_{1}+\beta\alpha_{3})e_{1}+\beta \alpha_{4} e_{4}+\alpha_{7} e_{6} \right),
\end{aligned}
\end{equation*}
where $c= (\alpha_{1}+\beta\alpha_{3})^2+\alpha_{2}^2+(\beta \alpha_{4})^2+\alpha_{7}^2$. If $\alpha_{2}=0,$ then it follows that $\gamma_{1}=1$ and thereby $\gamma_{2}=0,$ which is a contradiction. Otherwise, assume that $\alpha_{2}\neq0.$ Since $\alpha_{4}\neq 0,$ then $\alpha_{7}=\alpha_{1}=\beta=0.$ Hence, by substituting $f=e_{1}$ in \eqref{last} gives
\begin{align*}
e_{1}&=\gamma_{1}^{2}\pi_{V_{1}}e_{1}+\gamma_{3}^{2}\pi_{V_{3}}e_{1}
\\&=\gamma_{1}^{2}e_{1}+\gamma_{3}^{2}\pi_{\textnormal{span}\{\alpha_{3}e_{1}+\alpha_{4}e_{4}\}}e_{1}
\\&=\gamma_{1}^{2}e_{1}+\frac{\gamma_{3}^2}{\alpha_{3}^2+\alpha_{4}^2}\left\langle e_{1},\alpha_{3}e_{1}+\alpha_{4}e_{4} \right\rangle (\alpha_{3}e_{1}+\alpha_{4}e_{4})
\\&=\left( \gamma_{1}^{2}+\frac{\gamma_{3}^2\alpha_{3}^2}{\alpha_{3}^2+\alpha_{4}^2}\right) e_{1}+\frac{\gamma_{3}^2 \alpha_{3}\alpha_{4}}{\alpha_{3}^2+\alpha_{4}^2}e_{4}.
\end{align*}
It derives that $\alpha_{3}=0,$ as $\alpha_{4}\neq 0.$ Thus, $\gamma_{1}=1$ and thereby $\gamma_{2}=0,$ which is again a contradiction. Therefore, $\mathcal{V}$ is not weight-scalable. It is worthwhile to mention that the result can be obtained directly from Theorem \ref{30}.
\end{example}

In Example \ref{Counterexample}(1), a weight-scalable 2-excess fusion frame is provided for $\mathcal {H}_3$ that contains an orthogonal fusion Riesz basis. The subsequent result shows that the orthogonality of the Riesz part is a necessary condition for the weight-scalability of certain 2-excess fusion frames in $\mathcal {H}_3.$

\begin{corollary}\label{Riesz orthogonal}
Let the 2-excess fusion frame $\mathcal{V}$ given by \eqref{every 2 excess} be weight-scalable for $\mathcal {H}_3$. Then every fusion Riesz basis contained in $\mathcal{V}$ is orthogonal.
\end{corollary}

\begin{proof}
Suppose that $\mathcal{V}=\{V_{i}\}_{i=1}^{n}~(n \leq 3)$ is a $\gamma$-scalable 2-excess fusion frame containing a fusion Riesz basis $\mathcal{W}$. Then there exists a sequence of weights $\gamma=\{\gamma_{i}\}_{i=1}^{n}$ such that 
\begin{equation}\label{rty}
f=\sum_{i =1}^{n}\gamma_{i}^{2}\pi_{V_{i}}f, \quad  (f\in \mathcal {H}).
\end{equation}
As a result of the proof of Theorem \ref{main theorem gamma scalable}, the excess elements cannot be contained in only one subspace of $\mathcal{W},$ see also Corollary \ref{k-excess in one subspace}. By keeping this fact in mind, let $\{e_{i}\}_{i=1}^3$ be the canonical orthonormal basis for $\mathcal {H}_3.$ Notice that, Theorem \ref{one dim}(3) ensures that $\mathcal{W}$ must have more than two subspaces, as otherwise it is incomplete. Hence, $V_{i}$'s are represented as
\begin{align*}
V_{1}&= W_{1} \oplus \textnormal{span}\{x\},
\\
V_{2}&= W_{2} \oplus \textnormal{span}\{y\},
\\
V_{3}&= W_{3},
\end{align*}
where $x,y$ are the redundant elements of $\mathcal{V}.$ Then, according to Theorem \ref{one dim}(3) and without losing the generality, $V_{i}$'s can be rewritten as follows:
\begin{eqnarray}\label{in R3}
\begin{aligned}
V_{1}&= \textnormal{span}\{e_{1}\} \oplus \textnormal{span}\{\alpha e_{2}+\alpha' u\},
\\
V_{2}&= \textnormal{span}\{e_{2}\} \oplus \textnormal{span}\{ \beta e_{1}+\beta' u\},
\\
V_{3}&= \textnormal{span}\{u\},
\end{aligned}
\end{eqnarray}
where $ \alpha,\alpha',\beta,\beta' \in \Bbb C$ and $\alpha\beta \neq 0$. Moreover, the redundant vectors and $u=(u_{1},u_{2},u_{3})$ are unit vectors of $\mathcal {H}_3$ such that $u_{3}\neq 0.$ In light of \eqref{in R3}, we infer that 
\begin{equation}\label{casess}
\alpha'u_{1}=\beta'u_{2}=0.
\end{equation}
Furthermore, by Theorem \ref{one dim}(1) we have $\left\langle \alpha e_{2}+\alpha' u, \beta e_{1}+\beta' u \right\rangle =0,$ which yields $\alpha'\beta'=0.$ First, assume that $\alpha'=\beta'=0,$ then \eqref{in R3} is reformulated as follows: 
\begin{equation*}
V_{1}=V_{2}= \textnormal{span}\{e_{1},e_{2}\},~V_{3}= \textnormal{span}\{u\}.
\end{equation*}
Putting $f=e_{3}$ in \eqref{rty} leads 
\begin{align*}
e_{3}&=\gamma_{3}^{2}\pi_{V_{3}}e_{3}=\gamma_{3}^{2} \left\langle e_{3},u \right\rangle u 
\\&=\gamma_{3}^2\left(u_{1}u_{3}e_{1}+u_{2}u_{3}e_{2}+u_{3}^2e_{3}\right).
\end{align*}
Since $u_{3}\neq 0$, we get $u_{1}=u_{2}=0,$ that means $\mathcal{W}$ is orthogonal. Now, suppose that $\alpha'\neq 0$ and $\beta'=0,$ then $u_{1}=0,$ by \eqref{casess}. Thus, \eqref{in R3}, without losing the generality, is represented in the following manner:
\begin{eqnarray*}
\begin{aligned}
V_{1}&= \textnormal{span}\{e_{1}\} \oplus \textnormal{span}\{a e_{2}+b e_{3}\},
\\
V_{2}&= \textnormal{span}\{e_{1},e_{2}\},~V_{3}= \textnormal{span}\{u_{2}e_{2}+u_{3}e_{3}\},
\end{aligned}
\end{eqnarray*}
where $ a^{2}+b^{2}=u_{2}^{2}+u_{3}^{2}=1$ and $ b\neq 0.$ Substituting $f=e_{i}~(1\leq i \leq 3)$ in \eqref{rty} gives
\begin{align*}
\begin{cases}
(1)~\gamma_{1}^2+\gamma_{2}^2=1,
\\
(2)~a^2\gamma_{1}^2+\gamma_{2}^2+u_{2}^2\gamma_{3}^2=1,
\\
(3)~b^2\gamma_{1}^2+u_{3}^2\gamma_{3}^2=1,
\\
(4)~ab\gamma_{1}^2+u_{2}u_{3}\gamma_{3}^2=0.
\\
\end{cases}
\end{align*}
Summing equations (2) and (3) and using (1), we get $\gamma_{3}=1.$ If $u_{2}\neq 0,$ then by (2), (3) and (4) we obtain
\begin{align*}
\gamma_{2}^2=1-u_{2}^2-a^2\gamma_{1}^2&=u_{3}^2-\dfrac{a^2b^2\gamma_{1}^4}{b^2\gamma_{1}^2}
\\&=u_{3}^2-\dfrac{u_{2}^2u_{3}^2}{1-u_{3}^2}=u_{3}^2-u_{3}^2=0,
\end{align*}
a contradiction. Thus, $u_{2}=0,$ which implies that $\mathcal{W}$ is orthogonal. Also, in the case where $\alpha'=0$ and $\beta'\neq 0,$ the desired result is obtained by a similar argument.
\end{proof}

The next example shows that the converse of Corollary \ref{Riesz orthogonal} is not valid, in general. Obviously, the converse holds whenever $V_{1}=V_{2}$ in \eqref{in R3}.

\begin{example}
Let $\{e_{i}\}_{i=1}^3$ be the canonical orthonormal basis for $\mathcal {H}_3$ and $W_{i}=\textnormal{span}\{e_{i}\}$ for $1\leq i \leq 3.$ Then $\mathcal{V}=\left\lbrace W_{1} \oplus W_{2},W_{2} \oplus W_{3},W_{3}\right\rbrace $ is a 2-excess fusion frame containing the fusion orthonormal basis $\{W_{i}\}_{i=1}^3.$ However, $\mathcal{V}$ is not weight-scalable by Theorem \ref{30}.
\end{example}

\section{Weight-Scalability: General Case}

Finally, we would like to discuss the weight-scalability of fusion frames with higher excess. First of all, we note that tight fusion frames are clearly weight-scalable. Now, we consider fusion frames in which a specific subspace of a fusion Riesz basis is repeated. In the other words, let
\begin{equation*}
\mathcal{V}=\{W_{\ell},\ldots ,W_{\ell}\}\cup \{W_{i}\}_{i\in I}
\end{equation*}
be a fusion frame with the Riesz part $\mathcal{W}=\{W_{i}\}_{i \in I}$ and the redundant subspace $W_{\ell}~(\ell \in I)$ occurs $n$ times. Then, a direct computation shows that the following conditions are equivalent:
\begin{itemize}
\item[(i)] $\mathcal{V}$ is $\gamma$-scalable, $\gamma=\{\gamma_{i}\}_{i \in I}$.
\item[(ii)] $\mathcal{W}$ is $\omega$-scalable, $\omega_{i}=1,~i \in I.$
\item[(iii)] $\mathcal{W}$ is orthogonal,
\end{itemize}
where
\begin{equation*}
\gamma_{i}=\begin{cases}
1,& i\neq \ell,
\\
\frac{1}{\sqrt{n+1}},& i=\ell.
\end{cases}
\end{equation*}
This indicates that weights play an important role in the scalability of such fusion frames. The same result can be extended to the fusion frames with more redundant subspaces. Inspired by Theorem \ref{main theorem gamma scalable}, we now obtain an analogous result for the weight-scalability of $k$-excess fusion frames of the form 
\begin{equation}\label{gbh}
\mathcal{V}= \{V_{\ell}\}_{\ell =-k}^{-1} \cup \mathcal{W}, \quad (k\geq 2),
\end{equation}
where $\mathcal{W}=\{W_{i}\}_{i =1}^{\infty}$ is a fusion Riesz basis, $V_{\ell}=\textnormal{span}\{x^{(\ell)}\}$ and the redundant element $x^{(\ell)} =\sum_{i=1}^{\infty} x^{(\ell)}_{i}$ is a unit vector of $\mathcal {H}$ such that $x^{(\ell)}_{i} \in W_{i}$ for each $i \geq 1$. 

\begin{proposition}\label{1-excess general}
Let $\mathcal{V}$ be the $k$-excess fusion frame defined by \eqref{gbh} and $\mathcal{W}$ a fusion Riesz basis for $\mathcal {H}.$ Then the following conditions are equivalent:
\begin{itemize}
\item[(i)] $\mathcal{V}$ is $\gamma$-scalable, $\gamma=\{\gamma_{i}\}_{i=-k}^{\infty}$.
\item[(ii)] $\sum_{\ell =-k}^{-1}\gamma_{\ell}^{2}\pi_{W_{i}}S_{\mathcal{W}_{\gamma}}^{-1}\pi_{V_{\ell}}\pi_{W_{j}} =\left( \gamma_{i}^{-2}\delta_{i,j}-\pi_{W_{i}}\right)\pi_{W_{j}},$ for all $i , j \geq 1.$
\item[(iii)] $\sum_{\ell =-k}^{-1}\gamma_{\ell}^{2}\left\langle \pi_{W_{j}}f,x^{(\ell)} \right\rangle x^{(\ell)}_{i} =\left( \delta_{i,j}-\gamma_{i}^{2}\pi_{W_{i}}\right)\pi_{W_{j}}f,$
for all $f\in \mathcal {H}.$
\end{itemize}
\end{proposition}

It is worthy of note that by rearranging the excess elements among the subspaces of $\mathcal{W}$, one can achieve similar results to the above proposition. For instance, if the $k$-excess fusion frame $\mathcal{V}$ in Proposition \ref{1-excess general} is represented as
\begin{equation*}
\mathcal{V}= V_{0} \cup \mathcal{W},
\end{equation*}
where $V_{0}=\oplus_{\ell =-k}^{-1}V_{\ell}$, then (ii) and (iii) are derived as follows:
\begin{itemize}
\item[(ii)] $\gamma_{0}^{2}\pi_{W_{i}}S_{\mathcal{W}_{\gamma}}^{-1}\pi_{V_{0}}\pi_{W_{j}} =\left( \gamma_{i}^{-2}\delta_{i,j}-\pi_{W_{i}}\right)\pi_{W_{j}},$ for all $i , j \geq 1.$
\item[(iii)] $\gamma_{0}^{2}\sum_{\ell =-k}^{-1}\left\langle \pi_{W_{j}}f,x^{(\ell)} \right\rangle x^{(\ell)}_{i} =\left( \gamma_{i}^{2}\delta_{i,j}-\gamma_{i}^{2}\pi_{W_{i}}\right)\pi_{W_{j}}f,$
for all $f\in \mathcal {H}.$
\end{itemize}

The aim here is not to provide an exhaustive list of analogous conditions for the weight-scalability, as there are numerous situations in which the excess elements may be presented. The next result can be noted as a consequence of the proof of Theorem \ref{main theorem gamma scalable}.
\begin{corollary}\label{k-excess in one subspace}
Every $k$-excess fusion frame, whose all the redundant elements are in only one subspace of its Riesz part is not weight-scalable for any $k\geq 2$.
\end{corollary}

In the sequel, we focus on the weight-scalability of $k$-excess fusion frames whose the redundant vectors are located in more than one subspace of their Riesz part.

\begin{theorem}\label{non-orth. scalable equality}
Let $\mathcal{W}=\{W_{i}\}_{i=1}^{\infty}$ be a fusion Riesz basis for $\mathcal {H}.$ Take
\begin{align*}
\begin{cases}
V_{1}= W_{1}+ Z_{2},
\\
V_{2}= W_{2}+ Z_{1},
\end{cases}
\end{align*}
where $Z_{i}$ is a non zero closed subspace of $ W_{i}$ for $i=1,2$ and $\mathcal{V}=\left\lbrace V_{1},V_{2},W_{i}\right\rbrace _{i =3}^{\infty}$ constitutes a fusion frame for $\mathcal {H}.$ Then the following statements hold.
\begin{itemize}
\item [(i)] If $\mathcal{V}$ is weight-scalable, then $Z_{i}= W_{i}$ for $i=1,2$.

\item [(ii)] If $Z_{i}= W_{i}$ for $i=1,2,~\{W_{j}\}_{j=3}^{\infty}$ is orthogonal and $(W_{1}+W_{2})\perp W_{j}$ for all $j\geq 3$, then $\mathcal{V}$ is weight-scalable.

\end{itemize}
\end{theorem}

\begin{proof}
(i) Since $\mathcal{V}$ is $\gamma$-scalable, then there exists a sequence of weights $\gamma=\{\gamma_{i}\}_{i=1}^{\infty}$ such that \eqref{00100} holds. It implies that $\gamma_{1}^{2}+\gamma_{2}^{2}=1,$ as $V_{1}\cap V_{2}\neq\{0\}$. By contrary, assume that $Z_{1}\varsubsetneq W_{1}.$ It is not difficult to see that $W_{1}\cap Z_{1}^{\perp}\neq \{0\}.$ Hence, there exists a non zero vector $g \in W_{1}\cap Z_{1}^\perp,$ so that
\begin{align*}
g&=\gamma_{1}^{2}g+\gamma_{2}^{2}\pi_{V_{2}}g+\sum_{i =3}^{\infty}\gamma_{i}^{2}\pi_{W_{i}}g
\\&=\gamma_{1}^{2}g+\gamma_{2}^{2}\pi_{V_{2}}g,
\end{align*}
where the last equality is due to the Riesz decomposition property of $\mathcal {W}.$ Hence, $\pi_{V_{2}}g=g$ and thereby $g \in V_{2}=W_{2}+Z_{1}.$ Thus, there exist $w \in W_{2}$ and $z \in Z_{1}$ such that $g=w+z$. Since $g-z=w \in W_{1}\cap W_{2}$, we conclude that $g-z=0.$ So $g=z \in Z_{1},$ which is a contradiction. Therefore, $Z_{1}=W_{1}$ and by a similar argument $Z_{2}=W_{2}$ as well.

(ii) Suppose that $V_{i}= W_{1}+W_{2}$ for $i=1,2$. Taking $\gamma_{1},\gamma_{2}$ satisfying $ \gamma_{1}^{2}+\gamma_{2}^{2}=1$ and $\gamma_{j}=1$ for $j\geq 3$, we get
\begin{align*}
S_{\mathcal{V}_{\gamma}}&=(\gamma_{1}^{2}+\gamma_{2}^{2})\pi_{(W_{1}+W_{2})}+\sum_{j =3}^{\infty}\gamma_{j}^{2}\pi_{W_{j}}
\\&=\pi_{(W_{1}+W_{2})}+\sum_{j =3}^{\infty}\pi_{W_{j}}=\pi_{(W_{1}+W_{2})}+\pi_{\oplus_{j=3}^\infty W_{j}}= I_{\mathcal{H}}.
\end{align*}
\end{proof}

\begin{corollary}\label{orthogonality case}
Let $\mathcal{W}$ be an orthogonal fusion Riesz basis for $\mathcal{H}$. Then $\mathcal{V},$ introduced in Theorem \ref{non-orth. scalable equality} is weight-scalable if and only if $Z_{i}= W_{i}$ for $i=1,2$.
\end{corollary}

\begin{example}
Consider the 1-excess fusion frame $\mathcal{V}$ introduced in Example \ref{1-excess example}. Take $Z_{1}=\overline{\textnormal{span}} \{e_{i}\}_{i\geq -n}$ and $Z_{2}=\overline{\textnormal{span}} \{e_{i}\}_{i\leq m}$, for every $n,m \in \Bbb{N}$. Then $\mathcal{Z}=\{Z_{i}\}_{i=1}^2$ is a dual fusion frame of $\mathcal{V}$ with $e(\mathcal{Z})=n+m+1,$ see \cite{excess of fusion} for more details. Clearly, $\mathcal{Z}$ can be expressed as
\begin{align*}
\begin{cases}
Z_{1}= \overline{\textnormal{span}} \{e_{i}\}_{i\geq 1} \oplus \textnormal{span} \{e_{i}\}_{i=-n}^{0},
\\
Z_{2}= \overline{\textnormal{span}} \{e_{i}\}_{i< 1} \oplus \textnormal{span} \{e_{i}\}_{i=1}^{m}.
\end{cases}
\end{align*}
Thus, Corollary \ref{orthogonality case} assures that $\mathcal{Z}$ is weight-scalable if and only if $n=m=\infty.$
\end{example}

The following result states that under which conditions the $2$-excess fusion frame $\mathcal{V}$ given by \eqref{every 2 excess} is weight-scalable. The proof follows immediately from Theorem \ref{one dim}, Theorem \ref{non-orth. scalable equality} and Corollary \ref{orthogonality case}. 

\begin{corollary}\label{2-excess scalability}
Let $\mathcal{W}=\{W_{i}\}_{i=1}^\infty$ be a fusion Riesz basis for $\mathcal{H},~x\in W_{2}$ and $y\in W_{1}.$ If the $2$-excess fusion frame $\mathcal{V}$ defined by \eqref{every 2 excess} is weight-scalable, then $V_{1}=V_{2}=\textnormal{span}\{x,y\}.$  In particular, the converse holds if $\{W_{j}\}_{j=3}^{\infty}$ is orthogonal and $V_{1} \perp W_{j}$ for all $j\geq 3.$
\end{corollary}

Obviously, if $\mathcal{W}$ is orthogonal in the above corollary, then $\mathcal{V}$ is weight-scalable if and only if $V_{1}=V_{2}=\textnormal{span}\{x,y\}.$ Notice that Theorem \ref{one dim}(3) can be generalized to every $k$-excess fusion frame for $k >2$ whose the redundant vectors are in only two subspaces of its Riesz part.

\begin{theorem}\label{dim Wi}
Let $\mathcal{V}=\{V_{i}\}_{i=1}^{2} \cup \{W_{i}\}_{i=3}^{\infty}$ be a weight-scalable fusion frame for $\mathcal{H}$ such that $W_{i} \varsubsetneq V_{i}$ for $i=1,2$ and $\{W_{i}\}_{i=1}^{\infty}$ is a fusion Riesz basis. Then $W_{1}\perp W_{2}.$ Moreover, the following statements hold:
\begin{itemize}
\item[(i)] If $\gamma_{1}=1,$ then $\gamma_{2}=1,~V_{1}\perp W_{2}$ and $W_{2}\perp W_{i}$ for all $i\geq 3.$
\item[(ii)] If $\gamma_{2}=1,$ then $\gamma_{1}=1,~V_{2}\perp W_{1}$ and $W_{1}\perp W_{i}$ for all $i\geq 3.$
\item[(iii)] If $\gamma_{1}\neq 1~(\gamma_{2}\neq 1),$ then $\textnormal{dim}W_{1}\leq \textnormal{dim}(V_{2}\smallsetminus W_{2})~\big(\textnormal{dim}W_{2}\leq \textnormal{dim}(V_{1}\smallsetminus W_{1})\big)$.
\end{itemize}
\end{theorem}
\begin{proof}
Without losing the generality, we may choose subspaces $Z_{1}$ and $Z_{2}$ such that $\textnormal{dim} Z_{i}=\textnormal{dim}(V_{i}\smallsetminus W_{i})$ for $i=1,2$ and 
\begin{align*}
\begin{cases}
V_{1}= W_{1} \oplus Z_{1},
\\
V_{2}= W_{2} \oplus Z_{2}.
\end{cases}
\end{align*}
Furthermore, the first and second components of the elements of $Z_{1}$ and $Z_{2}$, respectively, are equal to zero in the Riesz decomposition.
Since $\mathcal{V}$ is $\gamma$-scalable, then there exists a sequence of weights $\gamma=\{\gamma_{i}\}_{i=1}^\infty$ such that \eqref{00100} holds. For every $ g\in W_{1}$ in \eqref{00100}, we have
\begin{equation}\label{main}
g=\gamma_{1}^{2}g+\gamma_{2}^{2}\pi_{W_{2}}g+\gamma_{2}^{2}\pi_{Z_{2}}g+\sum_{i =3}^{\infty}\gamma_{i}^{2}\pi_{W_{i}}g.
\end{equation}
Applying the Riesz decomposition property of $\mathcal{W},$ we infer that $\gamma_{2}^{2}\pi_{W_{2}}g=0$ and so $W_{1}\perp W_{2}.$ Now, for the ``moreover'' part:

(i) We note that \eqref{main} also holds for all $g\in V_{1}.$ Thus, if $\gamma_{1}=1,$ then the Riesz decomposition property of $\mathcal{W}$ implies that $V_{1}\perp W_{2}.$ Hence, putting a non zero vector $ f\in W_{2}$ in \eqref{00100}, we obtain $f=\gamma_{2}^{2}f+\sum_{i =3}^{\infty}\gamma_{i}^{2}\pi_{W_{i}}f.$ It follows that $\gamma_{2}=1$ and $W_{2}\perp W_{i}$ for all $i\geq 3.$  The proof of (ii) is analogous to that of (i).

(iii) If $\gamma_{1}\neq 1,$ then \eqref{main} gives $g=\gamma_{1}^{2}g+\gamma_{2}^{2}\pi_{Z_{2}}g.$ Hence, we get $g=\frac{\gamma_{2}^{2}}{1-\gamma_{1}^{2}}\pi_{Z_{2}}g,$ which ensures that
\begin{equation*}
\textnormal{dim}W_{1}\leq \textnormal{dim}Z_{2}=\textnormal{dim}(V_{2}\smallsetminus W_{2}).
\end{equation*}
In the case where $\gamma_{2}\neq 1,$ by means of an analogous argument, it is proved that $\textnormal{dim}W_{2}\leq \textnormal{dim}Z_{1}=\textnormal{dim}(V_{1}\smallsetminus W_{1}).$ This completes the proof.
\end{proof}

Suppose that $\{e_{i}\}_{i=1}^{\infty}$ is an orthonormal basis for $\mathcal {H},~W_{1}=\textnormal{span}\{e_{1}\}$ and $W_{2}=\textnormal{span}\{e_{2},e_{3}\}.$ Consider 
\begin{align*}
V_{1}&= W_{1} \oplus \overline{\textnormal{span}}\{e_{2},e_{3},e_{2i}+e_{2i+1}\}_{i=2}^{\infty},
\\
V_{2}&= W_{2} \oplus \overline{\textnormal{span}}\{e_{1},e_{2i}-e_{2i+1}\}_{i=2}^{\infty},
\\
W_{3}&= \textnormal{span}\{e_{i}\}_{i=4}^{\infty},
\end{align*}
where $\{W_{i}\}_{i=1}^{3}$ is a fusion orthonormal basis. Then $\mathcal{V}=\{V_{1},V_{2},W_{3}\}$ constitutes an infinite-excess 2-tight fusion frame for $\mathcal {H}.$ Hence, it follows that $\mathcal{V}$ is $\gamma$-scalable by $\gamma_{i}=\sqrt{1/2}$ for $1\leq i\leq 3.$ Obviously, $\textnormal{dim}W_{i}< \textnormal{dim}(V_{j}\smallsetminus W_{j})$ for distinct indexes $i,j=1,2,$ which confirms Theorem \ref{dim Wi}.

\textbf{Conflict-of-interest.} This work does not have any conflicts of interest.
 
\bibliographystyle{amsplain}

\end{document}